\documentclass[headsepline,paper=A4, 12pt, pagesize, english]{amsart}
\usepackage[a4paper,hscale=0.7,vscale=0.75,centering]{geometry}
\usepackage{color}
\usepackage{amssymb}
\usepackage{amsmath}
\usepackage{amsthm} 
\usepackage[english]{babel}
\usepackage[T1]{fontenc}
\usepackage[utf8]{inputenc}
\usepackage{tikz}
\usepackage{caption}
\usepackage{hyperref}  
 
\newtheorem{theorem}{Theorem}
\newtheorem{proposition}{Proposition}
\newtheorem{assumption}{Assumption}
\newtheorem{remark}{Remark}
\newtheorem{corollary}{Corollary}
\newtheorem{lemma}{Lemma}
 
\newcommand{\reals}{\mathbb{R}}
\newcommand{\naturals}{\mathbb{N}}
\newcommand{\norm}[1]{\left\lvert #1 \right\rvert}
\newcommand{\dnorm}[1]{\left\lvert\left\lvert #1 \right\rvert\right\rvert}
\newcommand{\sProd}[2]{\left\langle #1,#2 \right\rangle}
\newcommand{\generator}[1]{\mathcal{#1}} 
\newcommand{\wDist}{\mathcal{W}}
\newcommand{\basis}[1]{\textbf{e}_{#1}}

\newcommand{\constA}{\alpha}
\newcommand{\constB}{\beta}
\newcommand{\hilbert}{\mathbb{H}}
\newcommand{\G}{\mathcal{G}}
 
\begin{document} 
 
\title[Explicit contraction rates for a class of degenerate diffusions]{Explicit
contraction rates for a class of degenerate and infinite-dimensional diffusions}

\author{Raphael Zimmer}
\address{Institute for Applied Mathematics, University of Bonn, Endenicher Allee 60, 53115 Bonn, Germany}
\email{Raphael.Zimmer@uni-bonn.de}
 
\thanks{Financial support from the German Science foundation through the {\em Hausdorff Center for Mathematics} is gratefully acknowledged.}

\begin{abstract}
Given a separable and real Hilbert space $\hilbert$ and a 
	trace-class, symmetric and non-negative 
	operator $\G:\hilbert\rightarrow\hilbert$, we examine 
	the equation 
 	\begin{align*}  
 	dX_t	= -X_t\, dt + b(X_t) \, dt + \sqrt{2} \, dW_t, \qquad X_0=x\in\hilbert,
 	\end{align*}
 	where $(W_t)$ is a $\G$-Wiener process on $\hilbert$ and $b:\hilbert\rightarrow\hilbert$ is Lipschitz. We assume there is a splitting 
 	of $\hilbert$ into a
 	finite-dimensional space $\hilbert^l$ and its orthogonal complement
 	$\hilbert^h$ such that $\G$ is strictly positive definite on $\hilbert^l$ and the non-linearity
 	$b$ admits a contraction property on $\hilbert^h$. Assuming a geometric 
 	drift condition, we derive a Kantorovich ($L^1$ Wasserstein)
 	contraction with an \emph{explicit} contraction rate for the corresponding Markov kernels.
 	Our bounds on the rate are based on 
 	the eigenvalues of $\G$ on the space $\hilbert^l$, a Lipschitz bound on $b$ and a geometric drift condition. 
 	The results are derived using coupling methods.
\end{abstract} 

\maketitle 
\noindent
\begin{center}
{\scriptsize
Final version to appear in 'Stoch PDE: Anal Comp' is available at
Springer via \newline \url{http://dx.doi.org/10.1007/s40072-017-0091-8}.
}\end{center}

\section{Introduction} \label{secIntro}
Let $(\hilbert,\sProd{\cdot}{\cdot},\norm{\cdot})$ be a separable and real Hilbert space 
with inner product  $\sProd{\cdot}{\cdot}$ and norm $\norm{\cdot}$. 
Suppose that a trace-class, symmetric and non-negative operator
$\G:\hilbert \rightarrow \hilbert$ is given. Let $(\basis{k})_{k\in\naturals_+}$ be an orthonormal basis of $\hilbert$ 
such that for non-negative real numbers $(\lambda_k)$, we
have $\G \basis{k} = \lambda_k
\basis{k}$ and $\sum_{k=1}^{\infty}\lambda_k<\infty$, see e.g.\ \cite{MR751959} for the existence of such a basis.  Denote by $(W_t)$ a $\G$-Wiener process on $\hilbert$, i.e.
$W_t=\sum_{k=1}^\infty \sqrt{\lambda_k}\, B^k_t\, \basis{k}$ for independent Brownian motions $(B^k_t)$. We
consider the stochastic differential equation
\begin{align} \label{eqMain}
	dX_t	= - X_t\, dt+b(X_t)\, dt+ \sqrt{2} \, dW_t, \qquad X_0=x\in\hilbert,
\end{align}  
on the space $\hilbert$ and assume that the non-linearity $b:\hilbert\rightarrow \hilbert$ is
Lipschitz. In particular, there is a strong, non-explosive and continuous solution $(X_t)$ taking values in $\hilbert$, see e.g.\
\cite{MR757771}. Moreover, $(X_t)$ is a Feller process and we 
denote the Markov transition kernels by $(p_t)$. 
Given a probability measure $\mu$ on
$\hilbert$, we write $\mu p_t(dx)=\int p_t(y,dx)\, \mu(dy)$.
\par\medskip
 Equation \eqref{eqMain} has a
 	natural appearance in the domain of sampling problems and 
 	acts as a diffusion limit for Markov chain Monte Carlo (MCMC) methods,  
 	see   
 	\cite{MR2358638,MR2977981,arxiv1,MR2119934,MR3262508} and the references therein.
 	In particular, if $U:\hilbert\rightarrow\reals_+$ is a smooth function, if 
 	$\G$ is positive definite and if we choose the non-linearity
 	$b(x)=-\G \nabla_\hilbert U(x)$
 	in \eqref{eqMain}, then the results from \cite{MR2358638} imply that the Markov kernels $(p_t)$ admit a unique
 	invariant probability measure $\pi$ satisfying $\pi p_t=\pi$ for any $t\geq 0$. The measure $\pi$ 
 	is given by
 	\begin{eqnarray}\label{pidef}
 		\pi(dx) &\propto& \exp(-U(x))\ \mathcal{N}(0,\G)(dx),
 	\end{eqnarray}
 	where $\mathcal{N}(0,\G)$ denotes a centered normal distribution on $\hilbert$
 	with covariance opera\-tor $\G$.
 	Such measures appear for example in the area of diffusion bridges, cf.\ \cite{MR2358638}.
 	\par\medskip
 	Let $\mu$ be a given initial distribution. Given the outlined connection to sampling problems, an important question is whether
 	the measure $\mu p_t$ converges towards
 	$\pi$ for $t\rightarrow\infty$  
 	in some reasonable distance and how one can obtain 
 	explicit rates for the speed of convergence. We give conditions under which 
 	the convergence takes place in Kantorovich and $L^p$ Wasserstein distances at an exponential rate
 	and focus on establishing concrete bounds on the speed of convergence.
 	Inspired by the sampling setup, we work in the following setting: Fix $n\in\naturals_+$.
 	We consider a splitting of the Hilbert space $\hilbert$  into a space $\hilbert^l=\langle
\basis{1},\ldots,\basis{n} \rangle$, spanned by the first $n$ basis vectors,  
and its orthogonal complement $\hilbert^h$, i.e.\ $\hilbert=\hilbert^l\oplus\hilbert^h$. 
We call $\hilbert^l$ \emph{low}-dimensional 
and $\hilbert^h$ \emph{high}-dimensional space.
Given $x\in\hilbert$, we denote by $x^l$ and $x^h$ the orthogonal projections onto $\hilbert^l$ and $\hilbert^h$ respectively.
Our main assumptions are: 
\begin{assumption}\label{assLipDrift}
 There are constants $0\leq H_h<1$ and  $L_l,L_h, H_l\geq 0$ such that 
 \begin{eqnarray}
 		\norm{b^h(x)-b^h(y)} &\leq& H_l \norm{x^l-y^l} + H_h \norm{x^h-y^h} \quad\text{and}
 		 \\\norm{b^l(x)-b^l(y)} &\leq& \, L_l \norm{x^l-y^l} +\, L_h 	\norm{x^h-y^h} \quad\text{for any } x,y\in\hilbert. 
 	\end{eqnarray}
 \end{assumption} 
\begin{assumption}\label{assNonDeg}
 	$\G$ is strictly positive definite on $\hilbert^l$, i.e.\ for any $k\in\naturals$ with $1\leq k \leq n$, we have $\lambda_k>0$.
\end{assumption} 
 In the sampling setup described above, assuming that the map $x\mapsto\nabla U(x)$ is Lipschitz on $\hilbert$, it is always possible to find 
 a splitting $\hilbert=\hilbert^l\oplus\hilbert^h$
 such that Assumptions \ref{assLipDrift} and \ref{assNonDeg} 
 are satisfied, cf.\ Section \ref{secApplications}.
In addition to the above assumptions, we need some kind of localization argument, i.e.\ we assume either that $b$ is vanishing outside of a ball or
that a geometric drift condition holds, cf.\ Assumptions \ref{assLocalNonContractive} and \ref{assLypDri}
respectively. Based on these assumptions we derive quantitative Kantorovich contractions for the Markov kernels using coupling methods. 
The resulting 
contraction rates are given explicitly in terms of the eigenvalues of $\G$ on the space $\hilbert^l$, 
the constants from Assumption \ref{assLipDrift} and the localization argument.
\par\medskip
\textbf{Outline.}  
	The main results are presented in Section \ref{secStatements}. 
	The key statements are Theorem \ref{thmSecond} and Theorem \ref{thmThird}.
	The couplings are specified in Section \ref{secCoulings} and the proofs are given
	in Section \ref{secProofs}. Applications are considered
	in Section \ref{secApplications}.  In the remaining part of the introduction 
	we present additional motivation and references.
\par\medskip
The ergodicity of degenerate and infinite-dimensional models 
	has been extensively studied in the last two decades
	and by now there exists a comprehensive theory \cite{MR2478676,MR2773030,MR2786645,MR1417491}. 
	Huge parts of the theory have been developed trying to answer the question, whether the $2D$
	stochastic Navier-Stokes equation is uniquely ergodic in a hypoelleptic setting, where only a few dimensions 
	are stimulated directly 
	with noise, cf.\ \cite{MR2259251,MR2111726}. As an intermediate step
	to tackle the truly hypoelleptic setting, many authors 
	\cite{MR1933448,MR1937652,MR1910827,MR1845328,MR1785459,MR1868992,MR1930573,MR1868991}
	worked in an intermedium setting: 
	They considered a splitting of the underlying Hilbert space into a finite-dimensional space $\hilbert^u$
	of ``unstable modes'', where the dynamics is forced directly with noise, and an infinite-dimensional 
	complement $\hilbert^s$ of ``stable modes'', where the driving noise can be degenerate. Stable and unstable modes means in this context
	 that the long time behavior of the dynamics is determined 
	by the behavior on the space $\hilbert^u$,  cf.\ \cite{MR1939651}. 
	In this context,
	J.C.\ Mattingly proposed in \cite{MR1937652} a coupling approach to conclude exponential mixing 
	properties for the $2D$ stochastic Navier-Stokes equation. 
	In a related context, M.\ Hairer demonstrated in \cite{MR1939651} the strength of 
	\emph{asymptotic couplings} to show mixing
	properties of degenerate systems. Finally, J.C.\ Mattingly and M.\ Hairer were able to proof the
	unique ergodicity of the $2D$ Navier-Stokes equation in a hypoelleptic setting, which was a milestone
	in the development
	of ergodic theory for degenerate and infinite-dimensional systems \cite{MR2111726,MR2478676,MR2786645}.
	Embedding some of the key concepts of the theory into a uniformly applicable framework,   
 	Mattingly, Hairer and Scheutzow developed the weak Harris theorem \cite{MR2773030}. It can be interpreted 
 	as a generalization of classical Harris type theorems \cite{MR0084889,MR1287609,MR2894052,MR1423462,MR2857021}
 	which have become standard tools for proving geometric ergodicity of finite-dimensional Markov processes. The weak Harris
 	theorem further extends the range of possible applications and allows to establish geometric ergodicity under 
 	 verifiable conditions.
 Nevertheless, being a uniform framework, applicable for a large class of Markov processes,  
    the (weak) Harris theorem usually does not provide sharp constants for specific models and the resulting constants are often not connected to the
    structure of the  
    model in a transparent way. This is due to the fact that the corresponding Kantorovich distance is usually
    chosen in a somehow ad hoc way, cf.\ \cite{2016arXivE}.
    \par\medskip 
     In this work we do not have the aim of developing a uniform framework for various models. We focus on the 
     very specific model \eqref{eqMain} and establish Kantorovich contractions with explicit constants by 
    adapting the underlying Kantorovich distance in a very specific way to the
    structure of the model.
    The approach is based on a technique from \cite{MR2843007,eberle1}. Here, A.\ Eberle establishes Kantorovich contractions 
    with explicit constants
    for finite-dimensional and non-degenerate diffusions 
       using    
     a combination of reflection coupling \cite{MR841588} and concave distance functions.  While the
principle idea to study Kantorovich distances w.r.t.\ concave underlying distances
 occurred at other places in the literature before
 \cite{MR1345035,MR2773030}, 
 it is noteworthy that \cite{MR2843007,eberle1} presents a technique, how one can construct 
 an explicit concave distance function which, under some reasonable assumptions, maximizes the 
resulting contraction rate under the 
reflection coupling up to constant factors. Eberle's results are based on the assumption
     that the underlying deterministic system of the diffusion is contractive for ``large distances''. In the recent work
     \cite{2016arXivE} this assumption is replaced by a more general Lyapunov drift condition combining Lyapunov functions with concave distance functions and reflection coupling,
     partially motivated by \cite{MR2773030,MR2857021}. In this work, we use the main 
     ideas from \cite{MR2843007,eberle1,2016arXivE} and extend them to the infinite-dimensional and 
     possibly degenerate process \eqref{eqMain} by constructing an explicit \emph{asymptotic coupling} $(X_t,Y_t)$ of solutions to \eqref{eqMain}
     in the sense of \cite{MR1939651,MR1937652}, i.e.\ a coupling
     for which the processes $X_t$ and $Y_t$ converge to each other but do not necessarily meet in finite time.
     The Kantorovich contraction of the Markov kernels is then established by adapting the underlying cost function in a very specific way to the chosen coupling 
     and model.  
     \par\medskip 
     Up to the author's knowledge there are currently two works which use a reflection coupling
     to conclude exponential mixing properties of infinite-dimensional systems.
In \cite{MR2152380} a reflection coupling is used to prove
exponential convergence for a reaction-diffusion and Burgers equation driven by space-time white noise.
The article \cite{MR3316613} makes use of an ``approximated reflection coupling''
to derive gradient estimates and exponential mixing for a class of 
non-linear monotone SPDES, where the driving noise is a $\G$-Wiener process,
$\G$ being trace-class and satisfying $\sProd{x}{\G x}>0$ for any
$x\in\hilbert$. Moreover, it is assumed that the solution of the SPDE lies in the image of
$\G$, i.e.\ that the equation has some kind of smoothing properties. In both
articles exponential convergence in total variation norms is concluded.
In contrast to these settings, we allow the operator $\G$ to be degenerate on the infinite-dimensional space
$\hilbert^h$ and equation \eqref{eqMain} does not provide the
additional smoothing assumed in \cite{MR3316613}. Moreover, in our
setting it is in general not true that for arbitrary $x,y\in\hilbert$ we have $||\delta_x p_t-\delta_y
p_t||_{TV}\rightarrow 0$ for $t\rightarrow\infty$, see e.g.\ \cite[Example
  3.14]{MR2259251}. 
    
\section{Main Results}\label{secMain}
We present our main results.
In Section \ref{secStatements} 
we formulate the main statements. The coupling
approach leading to those statements is explained in Section \ref{secCoulings}. Finally,
the proofs are provided in Section \ref{secProofs}.
\subsection{Results} \label{secStatements}
We now formulate our contraction results. As a preparation, we first introduce
a norm $\norm{\cdot}_\alpha$ on $\hilbert$ which is equivalent to the Hilbert
space norm, but has the advantage that it puts additional weight on the
components in the space $\hilbert^h$. This enables us to exploit the contraction
property provided by Assumption \ref{assLipDrift}. We then formulate three Kantorovich
contractions with an increasing level of difficulty: In Proposition
\ref{thmFirst} we assume that the map $b$ is a contraction w.r.t.\ $\norm{\cdot}_\alpha$
and thus
a Kantorovich contraction can be established with ease.  In Theorem \ref{thmSecond}
we assume that $b$ is a contraction w.r.t.\ $\norm{\cdot}_\alpha$ only
for ``large distances'' and adapt the underlying metric of the Kantorovich
distance accordingly by involving a concave function. Finally, in Theorem
\ref{thmThird} we replace the contraction property for large distances by
a more general  geometric drift condition and combine the metric
considered in Theorem \ref{thmSecond} with Lyapunov functions.
\par\bigskip
Suppose that Assumption \ref{assLipDrift} holds true. 
Denote
 \begin{eqnarray}\label{defalpha}
 \constA &=& \frac{1+L_h}{1-H_h}\geq 1 \qquad\text{and}\qquad
\constB \ = \ \constA H_l+L_l-1.
\end{eqnarray}
We define a norm $\norm{\cdot}_\alpha$ on $\hilbert$, where the $\hilbert^h$
component is weighted by $\alpha$:
\begin{eqnarray*}  
\norm{x}_\alpha&=&\norm{x^l} \ + \  \alpha \; \norm{x^h}. 
\end{eqnarray*}
Observe that $\norm{\cdot}_\alpha$ is equivalent to $\norm{\cdot}$, i.e.\ for any $x\in\hilbert$,
\begin{eqnarray}\label{normineq}
	\norm{x} &\leq& \norm{x}_\alpha \ \leq \ \sqrt{2}\;\alpha \norm{x}.
\end{eqnarray} 
Assumption \ref{assLipDrift} implies
that the non-linearity $b$ is a contraction w.r.t.\ $\norm{\cdot}_\alpha$ in 
``certain regions of $\hilbert$''. More precisely, we have the following
statement:
\begin{lemma}\label{lemWeightedNorms}
	Assumption \ref{assLipDrift} implies the inequality 
	\begin{align}\label{lemWeightedNorms1}
	\norm{b(x)-b(y)}_\alpha\leq (1+\beta) \norm{x^l-y^l}_\alpha  + 
	\left(1-\frac{1}{\alpha}\right) \norm{x^h-y^h}_\alpha \quad \text{for any } x,y\in\hilbert.
	\end{align}
	Moreover, if $x,y\in\hilbert$ satisfy
	\begin{eqnarray}\label{contractiveSector}
		(1+\beta)\norm{x^l-y^l}&\leq& \frac{1}{2} \norm{x^h-y^h},
		\end{eqnarray}
		then it follows
		\begin{eqnarray}
		\norm{b(x)-b(y)}_\alpha &\leq& \left(1-\frac{1}{2\alpha}\right) \norm{x-y}_\alpha. 
	\end{eqnarray}
\end{lemma}
\begin{proof} 
Assumption \ref{assLipDrift} implies the inequalities
\begin{eqnarray*}
	\norm{b(x)-b(y)}_\alpha &=& \norm{b^l(x)-b^l(y)} \ + \ \alpha\; \norm{b^h(x)-b^h(y)}
\\&\leq&  (\alpha\;H_l+L_l)\; \norm{x^l-y^l} \ + \ \left(H_h+{L_h}/{\alpha}\right)
\alpha \norm{x^h-y^h} \\&=&(1+\beta) \; \norm{x^l-y^l}_\alpha \  + \ 
\left(1-{1}/{\alpha}\right) \; \norm{x^h-y^h}_\alpha.
\end{eqnarray*} 
If \eqref{contractiveSector} holds true, then we can further estimate:
	\begin{eqnarray*}
		\norm{b(x)-b(y)}_\alpha 
		&\leq& \norm{x^h-y^h}_\alpha - 1/2 \norm{x^h-y^h}
		\\&\leq&  \norm{x-y}_\alpha - \norm{x^l-y^l} -1/(2\alpha)\norm{x^h-y^h}_\alpha
		\\&\leq& \left(1-\min\left\{1/{(2\alpha),1}\right\}\right)\norm{x-y}_\alpha.
	\end{eqnarray*}
	  Since $\alpha\geq 1$, we conclude that
	$\min\left\{1/{(2\alpha),1}\right\} = 1/(2\alpha)$.
\end{proof}
Given a continuous function $d:\hilbert\times\hilbert\rightarrow\reals_+$,
the $L^1$ transportation cost of two Borel probability
measures $\mu$ and $\nu$ on $\hilbert$ w.r.t.\ the cost function $d$ is defined by
\begin{eqnarray*}
	\wDist_{d}(\nu,\mu) &= & \inf_{\gamma} \int~d(x,y)\ \gamma(dx\,dy),
\end{eqnarray*}
where the infimum is taken over all couplings $\gamma$ with 
marginals $\nu$ and $\mu$ respectively. If the function $d$ is a metric, then
$\wDist_{d}$
is called \emph{Kantorovich distance}. Let $\mathcal{P}$ be the set
of Borel probability measures on $\hilbert$ with finite first moment, i.e.\ $\int \norm{x} \mu(dx)<\infty$ for $\mu\in\mathcal{P}$. 
\par\smallskip
If $\beta<0$, then \eqref{lemWeightedNorms1} reveals that $b$ is a
\emph{contraction} on $\hilbert$ w.r.t.\ $\norm{\cdot}_\alpha$ which
implies the following trivial result.
\begin{proposition}\label{thmFirst}
Let Assumption \ref{assLipDrift} be true and $\beta<0$, then
\begin{eqnarray}\label{resEasyContraction}
\mathcal{W}_{d_1}(\mu p_t,\nu p_t)&\leq& e^{-c\,t}\ 
\mathcal{W}_{d_1}(\mu,\nu) \quad\text{for any
$\mu,\nu\in \mathcal{P}$ and $t\geq 0$,}
\end{eqnarray}
where the distance $d_1$ and the rate $c$ are given by
\begin{eqnarray*}
	d_1(x,y) &=&  \norm{x-y}_\alpha \quad\text{and}\quad c \ = \  \min\left\{ \alpha^{-1}, |\constB|\right\}.
\end{eqnarray*}
\end{proposition}
The assumption $\beta<0$ implies that the underlying deterministic system
of \eqref{eqMain} is contractive and thus the statement
even holds in the case $\G\equiv 0$. A proof using synchronous coupling is given in Section \ref{secProofs} for the readers convenience.
\par\smallskip
In order to tackle the case $\beta\geq 0$, we
demand that the noise in the space $\hilbert^l$ is non-degenerate, i.e.\ that Assumption \ref{assNonDeg} holds true.
Moreover, we assume that $b$ is a contraction w.r.t.\
$\norm{\cdot}_\alpha$ for ``large distances''. More precisely, we assume :
\begin{assumption}\label{assLocalNonContractive}
	There are $R\in(0,\infty)$ and $0\leq M<1$ such that
	\begin{eqnarray*}
		\norm{b(x)-b(y)}_\alpha & \leq & M \;\norm{x-y}_\alpha \quad\text{for any $x,y\in\hilbert$ with
	$\norm{x-y}_\alpha \geq  R$.}
	\end{eqnarray*}
\end{assumption}
The assumption is for example satisfied, if $b$ vanishes outside of a ball. 
Subsequently, we will replace Assumption \ref{assLocalNonContractive} by a more general geometric drift condition, 
cf.\ Assumption \ref{assLypDri}. 
Denote by $\lambda_\star=\min\{\lambda_k: k\in\naturals, 1\leq k \leq n \}$
the smallest eigenvalue of $\G$ on $\hilbert^l$. We present our
first main statement.
\begin{theorem}\label{thmSecond} 
	Let Assumption \ref{assLipDrift}, \ref{assNonDeg},  and
	\ref{assLocalNonContractive} be true and assume $\beta\geq 0$. There is a distance $d_2$ and a constant $c\in(0,\infty)$ such that
	\begin{eqnarray} 
		\wDist_{d_2}(\mu p_t, \nu p_t) &\leq& e^{-c\,t}\ \wDist_{d_2}(\mu,\nu) \quad\text{ for
	any } \mu,\nu\in\mathcal{P} \text{ and } t\geq 0. \label{eqThmScd} 
	\end{eqnarray}
	The rate $c$ is stated explicitly in \eqref{thm2c}. In the case $\beta>0$,
	a lower bound is given by 
	\begin{eqnarray}\label{clowerbound1} 
	c &\geq& \frac{1}{2}\,\exp\left(-\frac{\beta}{8\lambda_\star}R^2\right)\,
					\min\left\{\beta,1-M, \frac{1}{2\alpha}\right\}.
	\end{eqnarray}	
	The distance $d_2$ is equivalent to $\norm{\cdot}$ and is given by
	\begin{eqnarray*}
		d_2(x,y)&=&f\left(\norm{x-y}_\alpha\right),
	\end{eqnarray*}
	where
	$f:\reals_+\rightarrow\reals_+$ is a strictly increasing, concave and continuous function with $f(0)=0$. 
	The function is  explicitly defined in \eqref{thm2funcf}. It satisfies the relations
	\begin{eqnarray*}
		\frac{1}{2} &\leq& f'(r) \ \exp\left(\frac{\beta}{8\lambda_\star}r^2\right) \ \leq \ 
		1 \quad \text{ for } \;0<r<R \text{ and } \\
		f(r)&=&f(R)+\frac{1}{2}\exp\left(-\frac{\beta}{8\lambda_\star}R^2\right)(r-R)
		\quad \text{ for }r\geq R.
	\end{eqnarray*} 
\end{theorem}
Theorem \ref{thmSecond} extends ideas from \cite{MR2843007,eberle1} to an infinite-dimensional and degenerate setting 
using asymptotic couplings
in a similar spirit as \cite{MR1939651,MR1937652}. The proof is given in Section \ref{secProofs}.
 The occurring factors ${1}/{2}$ and
	${1}/{8}$ are, to some extend, 
	arbitrary.
	Notice that the degenerate case $\G|_{\hilbert^h}\equiv 0$ is covered by the statement. Given $p\geq 1$, we write 
	\begin{eqnarray*}
		\wDist^p(\mu,\nu) &=& \left(\inf_{\gamma} \int~\norm{x-y}^p\,\gamma(dx\,dy)\right)^{1/p}
	\end{eqnarray*} 
for the \emph{$L^p$ Wasserstein distance} of two measures
$\mu$ and $\nu$.  The Kantorovich contraction \eqref{eqThmScd} has several consequences.  Following \cite{eberle1}, we present some applications.
\begin{corollary}\label{corTHM21} 
	There is a unique invariant probability measure $\pi\in \mathcal{P}$ such that 
 	\begin{eqnarray}\label{thm2cor1}
		\wDist^1(\delta_x p_t, \pi) &\leq& 4\,\alpha\, e^{\frac{\beta R^2}{8
		\lambda_\star}}\,e^{-c\, t}\,\wDist^1(\delta_x, \pi) \quad\text{ for any $x\in
\hilbert$ and $t\geq 0$.}
	\end{eqnarray} 
 \end{corollary}
 For measurable $g:\hilbert\rightarrow\reals$, we denote the Lipschitz constant w.r.t.\ $d_2$ by 
 \begin{eqnarray}\label{lipnorm}
 	\dnorm{g}_{\operatorname{Lip}(d_2)}=\sup\left\{{\norm{g(x)-g(y)}}/{d_2(x,y)}
 :x,y\in \hilbert, x\not=y\right\}. 
 \end{eqnarray}
\begin{corollary}\label{corTHM23}
For any Lipschitz function $g$ and $t\geq 0$,
 \begin{eqnarray*}
 \sup\left\{\frac{\norm{(p_tg)(x)-(p_tg)(y)}}{\norm{x-y}} : x,y\in\hilbert,~
 x\not = y\right\} &\leq& \sqrt{2}\,\alpha \dnorm{g}_{\operatorname{Lip}(d_2)}\,
	 	e^{-c\,t}.
	 \end{eqnarray*}
\end{corollary} 
Further consequences are discussed after Theorem \ref{thmThird}.
\par\smallskip
We now generalize Theorem \ref{thmSecond} and replace Assumption \ref{assLocalNonContractive} by a
geometric drift condition using arguments related to the recent work \cite{2016arXivE}.  
Lyapunov drift conditions are widely used to study ergodicity and 
stability of Markov
processes, see e.g.\ \cite{MR1287609,MR2894052,MR2773030} and the references therein. Suppose that
a continuous function $V:\hilbert \rightarrow [1,\infty)$ is given for which the Fréchet
derivatives $\mathcal{D} V$ and $\mathcal{D}^2 V$ exist, are continuous and 
bounded in bounded subsets of $\hilbert$. Let 
\begin{eqnarray}\label{generator}
\mathcal{L}V(x) &=& \sProd{\mathcal{D} V(x)}{-x+b(x)} \ + \  \frac{1}{2}
	\sum_{k=1}^\infty \lambda_k \, \mathcal{D}^2 V(x)\, [\basis{k},\basis{k}].
\end{eqnarray}
\begin{assumption}\label{assLypDri} 
There are constants $C,\eta\in(0,\infty)$ such that for any $x\in \hilbert$,
\begin{eqnarray}\label{eqLypGenerator}
\mathcal{L}V(x) & \leq & C - \eta \; V(x).
\end{eqnarray}
Moreover, we assume that 
\begin{align*}
	&\lim_{r\rightarrow
\infty} \inf_{\norm{x}=r} V(x) \ = \  \infty & &\text{and} & &\theta \ = \
\sup_{x\in \hilbert} \frac{\norm{\mathcal{D} V(x)}}{V(x)} \ < \ \infty.
\end{align*}
\end{assumption}
	The condition  $\theta<\infty$ is imposed for simplicity and can be weakened. 
	We call a function $V$ satisfying the above conditions a \emph{Lyapunov function}. A typical candidate for a Lyapunov function is 
	$V(x)=1+\norm{x}^2$. Let
\begin{align}\label{defS}	
	S =   \left\{(x,y)\in\hilbert\times\hilbert:
	V(x)+V(y) < {8C}/{\eta}\right\} \quad\text{and}\quad R = \sup_{(x,y)\in S} \norm{x-y}_\alpha.
\end{align}
The set is chosen such that 
for any $(x,y)\not\in S$,
\begin{eqnarray}\label{genIneq}
	\mathcal{L} V(x)+\mathcal{L} V(y) &\leq& - ({\eta}/{2})\; \left(\;V(x)+V(y)\;\right) \ -\  2\;C.
\end{eqnarray}
Since $V$ is bounded from below,  the set $S$ cannot be empty and by continuity of $V$, $R>0$.
Moreover, Assumption \ref{assLypDri} implies that $R<\infty$. 
\par\smallskip
Let $\mathcal{P}_V$ be the set of probability measures $\mu$ on $\hilbert$
satisfying \- $\int V(x) \, \mu(dx)<\infty$ and write $\lambda^\star=\max\{\lambda_k: k\in\naturals, 1\leq k \leq n \}$ for
the largest eigenvalue of $\G$ on $\hilbert^l$. We call a continuous function $d:\hilbert\times\hilbert\rightarrow[0,\infty)$
a \emph{semimetric}, if it is symmetric and satisfies $d(x,y)=0$ if and only if
$x=y$.
We present our main result.
\begin{theorem} \label{thmThird}
Let Assumptions \ref{assLipDrift}, \ref{assNonDeg} and
	\ref{assLypDri} be true and assume $\beta\geq 0$. There is a semimetric $d_3$ and 
	a constant $c\in(0,\infty)$ such that 
	\begin{eqnarray} 
		\wDist_{d_3}(\mu p_t, \nu p_t) &\leq & e^{-c\,t}\ \wDist_{d_3}(\mu,\nu)\quad\text{ for
	any } \mu,\nu\in\mathcal{P}_V \text{ and } t\geq 0.
		\label{eqThm3rd}
	\end{eqnarray}
	The rate $c$ is given explicitly in \eqref{thm3rdContrRate}. If $\beta>0$, then a lower bound is given by
	\begin{eqnarray}\label{clowerbound2} 
c&\geq & \frac{1}{2} \ \min\left\{\ \exp\left(-\frac{\constB}{8\lambda_\star}
	R^2-2\theta\frac{\lambda^\star}{\lambda_\star}R\right)\ \min\left\{\ \frac{\beta}{2} \ , \ \frac{1}{4\alpha}\right\} \ , \ \eta \ \right\}.
	\end{eqnarray}
	The semimetric $d_3$ is given by
	\begin{eqnarray}\label{d3distance}
		d_3(x,y) & = & f\left(\norm{x-y}_\alpha\right)\left(1+ \epsilon \,V(x)+\epsilon \,V(y)\right),
	\end{eqnarray}
	 where $\epsilon\in(0,\infty)$ is a small constant. The function
	$f:\reals_+\rightarrow\reals_+$ is non-decreasing, concave and continuous with $f(0)=0$. It is
constant for $r\geq R$ and satisfies for $0<r<R$ the
 inequality
	\begin{eqnarray*}
		\frac{1}{2} & \leq &  f'(r) \  \exp\left(\frac{\constB}{8\lambda_\star}
	r^2+2\theta\frac{\lambda^\star}{\lambda_\star}r\right) \ \leq \ 
		1.
	\end{eqnarray*} 
	The explicit definitions of $f$ and $\epsilon$ are given
	in \eqref{thm3funcf} and \eqref{thm3rdContrRate} further below.
\end{theorem}
The extension of Theorem \ref{thmSecond} to the case of a geometric drift condition is in the same spirit as in the
related work \cite{2016arXivE}. 
The multiplicative structure of $d_3$ is inspired by \cite{MR2773030}. A proof of the theorem is given
in Section \ref{secProofs}. Notice that the function $d_3$ is in general not a metric,
since the triangle inequality might be violated. Nevertheless, as pointed out in 
\cite[Lemma 4.14]{MR2773030}, one can show that if the  Lyapunov function $V$
growths at most exponentially in $\norm{x}$, then $d_3$ satisfies a \emph{weak}
triangle inequality, i.e.\ 
there is $K\in(0,\infty)$ s.t. for all $x,y,z\in\hilbert$ we have $d_3(x,y)\leq K
[d_3(x,z)+d_3(z,y)]$. This is sufficient for several applications, as we discuss
now. The applications are well-known in the literature.
\begin{corollary} \label{thm3cor1}
Suppose that the assumptions of Theorem \ref{thmThird} hold true.
Let $p\geq 1$ and assume there is a constant $K\in(0,\infty)$ such that \-
$\norm{x-y}^p\leq K\left(V(x)+V(y)\right)$ for any $x,y\in\hilbert$. Then, the Markov kernels
$(p_t)$ admit a unique invariant probability measure $\pi\in \mathcal P_V$
such that for any $\mu\in \mathcal{P}_V$ and $t\geq 0$,
\begin{eqnarray*}
\wDist^p(\mu p_t, \pi)^p &\leq& 
2\,\exp\left(\frac{\constB}{8\lambda_\star}
	+2\theta\,\frac{\lambda^\star}{\lambda_\star}\right)
	 \max\left\{1,\frac{K}{\epsilon\;\min\{1,R\}}\right\}\,
		e^{-c\,t }\ \wDist_{d_3}(\mu , \pi ). 
\end{eqnarray*}
\end{corollary} 
If $\pi$ is symmetric w.r.t. $(p_t)$, which is for example the case
in the setting considered in Section \ref{sec31} further below, then
Corollary \ref{thm3cor1} implies a $L^2(\pi)$ spectral gap, 
cf.\ \cite[Proposition 2.8 and Theorem 2.15]{MR3262508} for a precise statement.
A Kantorovich contraction as in Theorem \ref{thmThird} has further
remarkable consequences: For example it allows to make statements about Markov
processes which are perturbations of $(X_t)$, cf.\ e.g.\ \cite[Section 4.1: Stability of
invariant measures]{MR2773030}. Moreover, it allows for
quantifications of bias and variances of ergodic averages, cf.\ 
\cite{MR2683634,eberle1,2016arXivE}. Since the latter sources do not provide statements which are directly
applicable in the setting of Theorem \ref{thmThird}, we formulate 
slightly adapted versions. Notice that similarly to
\eqref{lipnorm} we can define $||\cdot||_{\operatorname{Lip}(d_3)}$ for the semimetric $d_3$. 
\begin{corollary}\label{c1}
	Under the assumptions of Theorem \ref{thmThird}, it holds 
	\begin{eqnarray*}
 \sup\left\{\frac{\norm{(p_tg)(x)-(p_tg)(y)}}{\norm{x-y}} : x\not = y\right\} &\leq& \sqrt{2}\alpha \dnorm{g}_{\operatorname{Lip}(d_2)}
	 	\left(1+\epsilon V(x)+\epsilon V(y)\right) e^{-c\,t} 
	 \end{eqnarray*}
	 for any measurable function $g$ satisfying
	 $||g||_{\operatorname{Lip}(d_3)}<\infty$ and any $t\geq 0$.	
\end{corollary}
In particular, if $x\mapsto p_tg(x)$ is Fréchet differentiable at some point
$x\in\hilbert$, then Corollary \ref{c1} provides a bound on
$\norm{\nabla p_t g(x)}$. 
\begin{corollary}\label{c2}
	Under the assumptions of Corollary
	\ref{thm3cor1}, we have for any measurable function $g:\hilbert\rightarrow
	\reals$ with $||g||_{\operatorname{Lip}(d_3)}<\infty$, any $x\in\hilbert$ and
	$t>0$,
	\begin{eqnarray*}
		\norm{E_x\left[\frac{1}{t}\int_0^t g(X_s)\, ds - \int g d\pi \right]}
		&\leq & \frac{1-e^{-c\,t}}{c\, t} \, ||g||_{\operatorname{Lip}(d_3)}\, R\,
		(1+\epsilon \,V(x)+\epsilon\, C/\eta).
	\end{eqnarray*}
\end{corollary}
\begin{corollary}\label{c3}
Suppose that the assumptions of Theorem \ref{thmThird} hold true. Moreover, we
assume that the function $x\mapsto V(x)^2$ satisfies the geometric
drift condition
	\begin{eqnarray*}
		(\generator{L}V^2)(x) &\leq& C^\star - \eta^\star V(x)^2 \qquad\text{for
		any } x\in\reals^d,
	\end{eqnarray*} 
	with constants $C^\star,\eta^\star\in(0,\infty)$. It follows that
\begin{eqnarray}\label{decay}
\norm{\operatorname{Cov}_x[g(X_t),g(X_{t+h})]} &\leq&
\frac{3R^2}{2}||g||_{\operatorname{Lip}(d_3)}^2
(1+2\,\epsilon^2[C^\star/\eta^\star + e^{-\eta^\star t}V(x)^2]) e^{-c\,h}
\end{eqnarray}
for any measurable function $g$ satisfying
	 $||g||_{\operatorname{Lip}(d_3)}<\infty$ and any $t\geq 0$.
In particular, 
\begin{eqnarray*}
		\operatorname{Var}_x\left[\frac{1}{t}\int_0^t g(X_s)\, ds\right]&\leq&
		\frac{3\,R^2}{c\, t}  ||g||_{\operatorname{Lip}(d_3)}^2\,
		 \left(1+2\,\epsilon^2\left[C^\star/\eta^\star\,+\,e^{-\eta^\star t}\,
		V(x)^2\right]\right).
	\end{eqnarray*}	
\end{corollary}
The proofs of Corollaries \ref{c1}, \ref{c2} and \ref{c3} are nearly identical
to the ones given in \cite{eberle1,2016arXivE} and are not repeated here.
We remark that Theorem \ref{thmThird} can also be used to
make statements about the existence of solutions for the poisson equation
$-\generator{L}u=g$ associated with \eqref{eqMain} for a certain class of
functions $g$.
For a precise statement regarding this topic, we refer the reader to \cite[Theorem 3.1]{MR2933700}.

\subsection{Couplings}\label{secCoulings}    
We introduce the couplings used to derive upper bounds on the Kantorovich 
distances occurring in Proposition \ref{thmFirst}, Theorem \ref{thmSecond} and
Theorem \ref{thmThird}.
\subsubsection{Synchronous coupling}\label{secCouplSynch}
Fix initial values $(x_0,y_0)\in \hilbert \times \hilbert$. We call $(X_t,Y_t)$ a 
\emph{synchronous coupling}, if it is a solution of the equation
	\begin{eqnarray*}
			dX_t &=& - X_t \ dt \ + \ b(X_t)\ dt \ + \ \sqrt{2} \ dW_t,\\
			dY_t &=& - Y_t \,\ dt \,\ + \ b(Y_t)\,\ dt \ + \ \sqrt{2} \ dW_t, \qquad (X_0,Y_0) \ = \ (x_0,y_0),
	\end{eqnarray*} 
	on the space $\hilbert\oplus\hilbert$, where $(W_t)$ is a $\G$-Wiener process on $\hilbert$. The coupling is well-known and 
	used to prove Proposition \ref{thmFirst}.
\subsubsection{Reflection coupling for non-degenerate and finite-dimensional diffusions}\label{secCouplRefl}
In order to explain the coupling leading to  Theorem \ref{thmSecond} and
Theorem \ref{thmThird}, we shortly recall reflection coupling for non-degenerate and finite-dimensional diffusions, which goes back to \cite{MR841588}.
We consider the following SDE in $\reals^d$:
\begin{eqnarray}
	dR_t &=& a(R_t)\ dt \ + \ \sigma \ dB_t,
\end{eqnarray}
where $a:\reals^d\rightarrow\reals^d$ is (say) Lipschitz, $\sigma\in\reals^{d\times d}$ satisfies $\det(\sigma)>0$
and $(B_t)$ is a $d$-dimensional Brownian motion. A reflection coupling $(R_t,S_t)$ starting at $(r_0,s_0)\in\reals^{2d}$ is a solution
of the equation
\begin{eqnarray*}
			dR_t &=& a(R_t)\ dt \ + \ \sigma \ dB_t, \quad (R_0,S_0) \ = \ (r_0,s_0),\\
			dS_t &=& a(S_t)\ dt \ + \ \sigma\, \left(\operatorname{I}_d - 2\,\frac{\sigma^{-1} (R_t-S_t)}{\norm{\sigma^{-1} (R_t-S_t)}}\sProd{\frac{\sigma^{-1} (R_t-S_t)}{\norm{\sigma^{-1} (R_t-S_t)}}}{\cdot}\right) \ dB_t, 
			\quad t<T\\
			S_t &=& R_t, \quad t\geq T, 
\end{eqnarray*} 
where $T=\inf\{t\geq 0: X_t=Y_t\}$ is the coupling time.
One of the crucial properties of reflection coupling is that $r_t=\norm{R_t-S_t}$ satisfies almost surley 
the equation
\begin{eqnarray*}
	d\norm{r_t} &=& r_t^{-1} \sProd{R_t-S_t}{a(R_t)-a(S_t)} dt  
	+  2 \norm{\sigma^{-1}(R_t-S_t)}^{-1} r_t \ dW_t, \quad t<T,
\end{eqnarray*}
 where $(W_t)$ is a one-dimensional Brownian motion.
 We see that the driving noise $(W_t)$ has a direct impact on $\norm{R_t-S_t}$, see \cite{eberle1} for details.

\subsubsection{Switching between reflection and synchronous coupling}\label{seccoupl3}
We present the coupling used to prove Theorem \ref{thmSecond} and Theorem \ref{thmThird}.
Before we introduce the coupling in a rigorous way, 
we shortly explain the strategy: Let $(X_t,Y_t)$ be a \emph{synchronous coupling} of solutions to \eqref{eqMain},
i.e.\ let the processes $(X_t)$ and $(Y_t)$ be driven by the same noise. We argue pathwise. 
Assume that $X_t-Y_t$ satisfies for some $t\geq 0$ the inequality
\begin{eqnarray}\label{sectorcond}
	H_l \norm{X_t^l-Y_t^l} &\leq& (1-H_h)  \norm{X_t^h-Y_t^h}/2, 
\end{eqnarray}
then Assumption \ref{assLipDrift} implies that
\begin{eqnarray*} 
	\norm{b^h(x)-b^h(y)}&\leq& H_l \norm{X_t^l-Y_t^l} \ +\  H_h \norm{X_t^h-Y_t^h} \ \leq \ (1+H_h) \norm{X_t^h-Y_t^h}/2,
\end{eqnarray*}
where $(1+H_h)/2<1$ by assumption. In particular, as long as $X_t-Y_t$ satisfies \eqref{sectorcond}, 
$\norm{X^h_t-Y^h_t}$ decreases exponentially fast, while $\norm{X^l_t-Y^l_t}$ might increase at the same time. 
At some point, as time increases, $X_t-Y_t$ might not satisfy  \eqref{sectorcond} any more. 
The idea is now to use a reflection
coupling of $X_t^l$ and $Y_t^l$ in the space $\hilbert^l$ with the aim of 
decreasing $\norm{X_t^l-Y_t^l}$ relative to $\norm{X_t^h-Y_t^h}$. 
As soon as $X_t-Y_t$ satisfies again \eqref{sectorcond}, we switch the coupling to a \emph{synchronous coupling} and wait for a decrease of $\norm{X_t^h-Y_t^h}$. 
If $\norm{X_t^h-Y_t^h}$ gets again
	``small'' compared to $\norm{X_t^l-Y_t^l}$, we switch to a \emph{reflection
	coupling} in $\hilbert^l$ and so an and so forth. The coupling is visualized in Figure \ref{figure1}.
	As remarked above, during the phases $X_t-Y_t$ satisfies \eqref{sectorcond}, $\norm{X^l_t-Y^l_t}$ might increase.
	In order to see a contraction of $X_t-Y_t$, we measure the distance with the 
	weighted norm $\norm{\cdot}_\alpha$ and replace 
	the sector condition \eqref{sectorcond} by \eqref{contractiveSector} provided by Lemma \ref{lemWeightedNorms}.
	Indeed, as long as $X_t-Y_t$ satisfies  \eqref{contractiveSector}, $\norm{X_t-Y_t}_\alpha$ decreases exponentially fast.
	This is of course not true, if $X_t-Y_t$ fails to satisfy \eqref{contractiveSector}. Nevertheless, in the setting of Theorem \ref{thmSecond},
	an exponential decay of $f(\norm{X_t-Y_t}_\alpha)$ still holds \emph{on average}, 
	if we use an appropriate concave function $f$ following \cite{eberle1}. 
	The coupling strategy is similar to the ones from \cite{MR1939651,MR1937652}: We identify a region 
	where the deterministic system corresponding to \eqref{eqMain}
	has a contraction property and then use the available noise to drive the coupling
	into those regions.
	\begin{figure}
\centering
\begin{tikzpicture}[
      scale=1.6,
      a/.style={thick,<->,shorten >=2pt,shorten <=2pt,>=stealth}
    ] 
    \draw [<->,thick] (0,2) node (yaxis) [above] {$\norm{X_\cdot^l-Y_\cdot^l}$}
        |- (3,0) node (xaxis) [right] {$\norm{X_\cdot^h-Y_\cdot^h}$};
    \draw (0,0) coordinate (a_1) -- (3,1) coordinate (a_2) node[right] {};
 
    \coordinate [label={right: \small{$X_0-Y_0$}}] (c) at (2.5,1.5);
       
     \coordinate [label={right:A}] (E) at (1,1.3); 
   	 \coordinate [label={right:B}] (E) at (3,0.4); 
   	 \coordinate [label={right:Region A}] (F) at (5,2);
    \draw[a] (5.4,1.8) -- (5.4,1.4) node{};
    \coordinate [label={right:\tiny{reflection coupling in $\hilbert^l$}}] (G) at (5.65,1.6);
     \draw[a] (5.2,1.2) -- (5.6,1.2) node{}; 
        \coordinate [label={right:\tiny{synchronous coupling in $\hilbert^h$}}] (H) at (5.65,1.2);
         	 \coordinate [label={right:Region B}] (F) at (5,0.8);
    \draw[a] (5.4,0.6) -- (5.4,0.2) node{}; 
     \draw[a] (5.2,0.4) -- (5.6,0.4) node{}; 
        \coordinate [label={right:\tiny{synchronous coupling in $\hilbert$}}] (H) at (5.65,0.4);
        
     \fill[red] (c) circle (1pt) ;
    \coordinate [label={[label distance=0.25cm]right: \small{$X_t-Y_t$}}] (d) at (1.65,0.35); \fill[blue] (d) circle (1pt) ;
      
       \draw [red] plot coordinates  
     {(2.5,1.52)(2.5,1.45)(2.49,1.44)(2.49,1.45)(2.49,1.45)(2.49,1.47)(2.48,1.48)(2.48,1.45)(2.48,1.5)(2.48,1.5)(2.47,1.49)(2.47,1.48)(2.47,1.46)(2.47,1.46)(2.46,1.42)(2.46,1.41)(2.46,1.3900000000000001)(2.46,1.3800000000000001)(2.45,1.3900000000000001)(2.45,1.4000000000000001)(2.45,1.37)(2.45,1.34)(2.44,1.3900000000000001)(2.44,1.34)(2.44,1.32)(2.44,1.35)(2.43,1.28)(2.43,1.31)(2.43,1.29)(2.43,1.33)(2.42,1.34)(2.42,1.36)(2.42,1.29)(2.42,1.32)(2.41,1.24)(2.41,1.25)(2.41,1.28)(2.41,1.31)(2.4,1.33)(2.4,1.3)(2.4,1.27)(2.4,1.27)(2.39,1.25)(2.39,1.29)(2.39,1.27)(2.39,1.26)(2.39,1.27)(2.38,1.27)(2.38,1.3)(2.38,1.25)(2.38,1.28)(2.37,1.29)(2.37,1.35)(2.37,1.33)(2.37,1.34)(2.36,1.29)(2.36,1.3)(2.36,1.24)(2.36,1.26)(2.35,1.26)(2.35,1.27)(2.35,1.31)(2.35,1.26)(2.34,1.33)(2.34,1.37)(2.34,1.43)(2.34,1.37)(2.34,1.3900000000000001)(2.33,1.43)(2.33,1.41)(2.33,1.3900000000000001)(2.33,1.3800000000000001)(2.32,1.37)(2.32,1.36)(2.32,1.35)(2.32,1.4000000000000001)(2.31,1.42)(2.31,1.44)(2.31,1.3800000000000001)(2.31,1.33)(2.31,1.35)(2.3000000000000003,1.37)(2.3000000000000003,1.36)(2.3000000000000003,1.37)(2.3000000000000003,1.41)(2.29,1.44)(2.29,1.46)(2.29,1.44)(2.29,1.41)(2.2800000000000002,1.42)(2.2800000000000002,1.45)(2.2800000000000002,1.45)(2.2800000000000002,1.47)(2.2800000000000002,1.48)(2.27,1.48)(2.27,1.47)(2.27,1.45)(2.27,1.36)(2.2600000000000002,1.35)(2.2600000000000002,1.35)(2.2600000000000002,1.35)(2.2600000000000002,1.32)(2.2600000000000002,1.31)(2.25,1.29)(2.25,1.27)(2.25,1.27)(2.25,1.26)(2.24,1.28)(2.24,1.28)(2.24,1.28)(2.24,1.3)(2.23,1.24)(2.23,1.29)(2.23,1.29)(2.23,1.27)(2.23,1.23)(2.22,1.25)(2.22,1.25)(2.22,1.28)(2.22,1.27)(2.21,1.28)(2.21,1.27)(2.21,1.29)(2.21,1.27)(2.21,1.24)(2.2,1.28)(2.2,1.3)(2.2,1.23)(2.2,1.24)(2.2,1.17)(2.19,1.17)(2.19,1.2)(2.19,1.2)(2.19,1.2)(2.18,1.11)(2.18,1.05)(2.18,1.11)(2.18,1.1300000000000001)(2.18,1.06)(2.17,1.12)(2.17,1.1400000000000001)(2.17,1.19)(2.17,1.1300000000000001)(2.16,1.11)(2.16,1.2)(2.16,1.18)(2.16,1.21)(2.16,1.18)(2.15,1.2)(2.15,1.2)(2.15,1.24)(2.15,1.2)(2.15,1.2)(2.14,1.18)(2.14,1.1500000000000001)(2.14,1.09)(2.14,1.11)(2.13,1.1)(2.13,1.06)(2.13,1.04)(2.13,1.08)(2.13,1.1)(2.12,1.07)(2.12,1.05)(2.12,1.04)(2.12,1.02)(2.12,1.07)(2.11,1.04)(2.11,1.02)(2.11,1.05)(2.11,1.07)(2.1,1.07)(2.1,1.09)(2.1,1.05)(2.1,1.08)(2.1,1.1)(2.09,1.12)(2.09,1.16)(2.09,1.17)(2.09,1.18)(2.09,1.19)(2.08,1.17)(2.08,1.18)(2.08,1.1400000000000001)(2.08,1.11)(2.08,1.11)(2.07,1.1500000000000001)(2.07,1.1500000000000001)(2.07,1.1)(2.07,1.11)(2.07,1.1)(2.06,1.06)(2.06,1.02)(2.06,1.02)(2.06,1.01)(2.05,0.91)(2.05,0.79)(2.05,0.8300000000000001)(2.05,0.8300000000000001)(2.05,0.81)(2.04,0.9)(2.04,0.89)(2.04,0.91)(2.04,0.93)(2.04,0.93)(2.0300000000000002,0.91)(2.0300000000000002,0.9)(2.0300000000000002,0.85)(2.0300000000000002,0.8)(2.0300000000000002,0.78)(2.02,0.74)(2.02,0.77)(2.02,0.78)(2.02,0.78)(2.02,0.77)(2.0100000000000002,0.79)(2.0100000000000002,0.75)(2.0100000000000002,0.71)(2.0100000000000002,0.75)(2.0100000000000002,0.71)(2.0,0.76)(2.0,0.65)};
      
      \draw [blue] plot [smooth] coordinates {(2.0,0.65) (1.82,0.45) (1.75,0.55)
     (1.65,0.35)  (1.55,0.25)  (1.45,0.48)};
        \draw [red] plot [smooth] coordinates 
        {(1.45,0.48)(1.45,0.52)(1.45,0.56)(1.45,0.6)(1.44,0.56)(1.44,0.59)(1.44,0.61)(1.44,0.6)(1.44,0.63)(1.44,0.62)(1.44,0.61)(1.43,0.66)(1.43,0.62)(1.43,0.61)(1.43,0.61)(1.43,0.59)(1.43,0.63)(1.43,0.64)(1.42,0.61)(1.42,0.64)(1.42,0.62)(1.42,0.68)(1.42,0.73)(1.42,0.75)(1.42,0.75)(1.41,0.71)(1.41,0.72)(1.41,0.73)(1.41,0.6900000000000001)(1.41,0.65)(1.41,0.6900000000000001)(1.41,0.66)(1.4000000000000001,0.63)(1.4000000000000001,0.64)(1.4000000000000001,0.64)(1.4000000000000001,0.68)(1.4000000000000001,0.64)(1.4000000000000001,0.63)(1.4000000000000001,0.5700000000000001)(1.3900000000000001,0.54)(1.3900000000000001,0.51)(1.3900000000000001,0.49)(1.3900000000000001,0.54)(1.3900000000000001,0.55)(1.3900000000000001,0.56)(1.3900000000000001,0.5)(1.3800000000000001,0.53)(1.3800000000000001,0.51)(1.3800000000000001,0.47000000000000003)(1.3800000000000001,0.51)(1.3800000000000001,0.46)};
            \draw [blue] plot [smooth] coordinates {(1.3800000000000001,0.46) (1,0.23)
            (0.75,0.25)};
         \draw [red] plot [smooth] coordinates
         {  (0.75,0.25)
         (0.75,0.37)(0.75,0.38)(0.74,0.37)(0.74,0.37)(0.74,0.39)(0.74,0.39)(0.73,0.4)(0.73,0.36)(0.73,0.29)(0.73,0.3)(0.73,0.34)(0.72,0.34)(0.72,0.34)(0.72,0.4)(0.72,0.45)(0.71,0.41000000000000003)(0.71,0.42)(0.71,0.44)(0.71,0.42)(0.71,0.45)(0.7000000000000001,0.46)(0.7000000000000001,0.47000000000000003)(0.7000000000000001,0.45)(0.7000000000000001,0.46)(0.7000000000000001,0.44)(0.6900000000000001,0.45)(0.6900000000000001,0.43)(0.6900000000000001,0.43)(0.6900000000000001,0.41000000000000003)(0.6900000000000001,0.44)(0.68,0.4)(0.68,0.37)(0.68,0.39)(0.68,0.34)(0.68,0.28)(0.67,0.33)(0.67,0.33)(0.67,0.35000000000000003)(0.67,0.31)(0.67,0.27)(0.66,0.29)(0.66,0.33)(0.66,0.32)(0.66,0.31)(0.66,0.34)(0.65,0.37)(0.65,0.4)(0.65,0.41000000000000003)(0.65,0.41000000000000003)(0.65,0.41000000000000003)(0.64,0.38)(0.64,0.41000000000000003)(0.64,0.41000000000000003)(0.64,0.49)(0.64,0.5)(0.63,0.51)(0.63,0.52)(0.63,0.51)(0.63,0.53)(0.63,0.5)(0.62,0.48)(0.62,0.37)(0.62,0.41000000000000003)(0.62,0.39)(0.62,0.33)(0.62,0.31)(0.61,0.28)(0.61,0.28)(0.61,0.36)(0.61,0.33)(0.61,0.36)(0.6,0.35000000000000003)(0.6,0.41000000000000003)(0.6,0.39)(0.6,0.41000000000000003)(0.6,0.43)(0.6,0.4)(0.59,0.41000000000000003)(0.59,0.41000000000000003)(0.59,0.38)(0.59,0.32)(0.59,0.35000000000000003)(0.58,0.38)(0.58,0.37)(0.58,0.36)(0.58,0.37)(0.58,0.37)(0.58,0.37)(0.5700000000000001,0.35000000000000003)(0.5700000000000001,0.34)(0.5700000000000001,0.29)(0.5700000000000001,0.27)(0.5700000000000001,0.27)(0.5700000000000001,0.31)(0.56,0.33)(0.56,0.29)(0.56,0.3)(0.56,0.27)(0.56,0.28)(0.56,0.3)(0.55,0.33)(0.55,0.34)(0.55,0.32)(0.55,0.33)(0.55,0.41000000000000003)(0.55,0.41000000000000003)(0.54,0.39)(0.54,0.38)(0.54,0.35000000000000003)(0.54,0.33)(0.54,0.32)(0.54,0.32)(0.53,0.36)(0.53,0.36)(0.53,0.37)(0.53,0.37)(0.53,0.41000000000000003)(0.53,0.4)(0.52,0.44)(0.52,0.38)(0.52,0.32)(0.52,0.31)(0.52,0.25)(0.52,0.24)(0.52,0.26)(0.51,0.3)(0.51,0.25)(0.51,0.23)(0.51,0.23)(0.51,0.21)(0.51,0.21)(0.5,0.23)(0.5,0.22)(0.5,0.22)(0.5,0.3)(0.5,0.31)(0.5,0.3)(0.5,0.26)(0.49,0.31)(0.49,0.31)(0.49,0.27)(0.49,0.27)(0.49,0.28)(0.49,0.25)(0.49,0.22)(0.48,0.21)(0.48,0.23)(0.48,0.25)(0.48,0.23)(0.48,0.29)(0.48,0.24)(0.48,0.3)(0.47000000000000003,0.27)(0.47000000000000003,0.31)(0.47000000000000003,0.29)(0.47000000000000003,0.26)(0.47000000000000003,0.25)(0.47000000000000003,0.29)(0.47000000000000003,0.29)(0.46,0.25)(0.46,0.26)(0.46,0.22)(0.46,0.21)(0.46,0.25)(0.46,0.26)(0.46,0.26)(0.45,0.26)(0.45,0.28)(0.45,0.29)(0.45,0.29)(0.45,0.36)(0.45,0.31)(0.45,0.3)(0.44,0.35000000000000003)(0.44,0.32)(0.44,0.24)(0.44,0.25)(0.44,0.25)(0.44,0.24)(0.44,0.27)(0.44,0.29)(0.43,0.32)(0.43,0.35000000000000003)(0.43,0.36)(0.43,0.32)(0.43,0.31)(0.43,0.27)(0.43,0.28)(0.43,0.29)(0.42,0.3)(0.42,0.31)(0.42,0.31)(0.42,0.32)(0.42,0.32)(0.42,0.33)(0.42,0.37)(0.41000000000000003,0.38)(0.41000000000000003,0.34)(0.41000000000000003,0.31)(0.41000000000000003,0.34)(0.41000000000000003,0.28)(0.41000000000000003,0.29)(0.41000000000000003,0.32)(0.41000000000000003,0.29)(0.41000000000000003,0.29)(0.4,0.26)(0.4,0.25)(0.4,0.28)(0.4,0.22)(0.4,0.24)(0.4,0.21)(0.4,0.22)(0.4,0.22)(0.39,0.22)(0.39,0.25)(0.39,0.16)(0.39,0.12)}; 
         \draw [blue] plot [smooth] coordinates {(0.39,0.12) (0.2,0.03)
           (0,0)};
    \end{tikzpicture}
\caption{Asymptotic coupling\label{figure1}}
\end{figure}
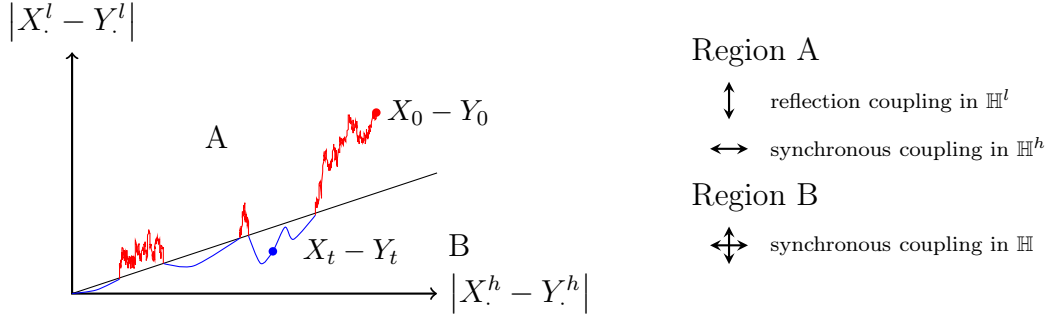
\par\medskip
We now define the coupling in a rigorous way. Fix small $\delta>0$ and denote
\begin{eqnarray}\label{defSectors}
	\mathcal{S}_{SC} &=&\left\{x\in \hilbert:
	4\left(\beta+1\right)\norm{x^l}\leq
	\norm{x^h}\right\}\cup\left\{x\in \hilbert:
	\norm{x}_\alpha\leq {\delta}/{2} \right\},
	\\\mathcal{S}_{RC}&=&\left\{x\in \hilbert:
	2\left(\beta+1\right)\norm{x^l}\geq
	\norm{x^h}\right\}\cap\left\{x\in \hilbert:
	\norm{x}_\alpha\geq \delta \right\}.\nonumber
\end{eqnarray}
In comparison to the informal description further above, we add transition regions
to realize transitions between the different coupling types.
We describe the coupling first in words: The driving noise in the subspace $\hilbert^h$ is always coupled
\emph{synchronously}, i.e.\ the same noise is used to drive $X^h_t$ and $Y^h_t$.
 In the finite-dimensional
subspace $\hilbert^l$ we use a \emph{reflection coupling} of the driving noise
if $X_t-Y_t\in \mathcal{S}_{RC}$ and a \emph{synchronous coupling} if
$X_t-Y_t\in \mathcal{S}_{SC}$. The definition of the above sets is motivated by
Lemma \ref{lemWeightedNorms}. The two sets $\mathcal{S}_{RC}$ and
$\mathcal{S}_{SC}$ are closed, disjoint and $\inf_{x\in\mathcal{S}_{RC}, y\in\mathcal{S}_{SC}} \norm{x-y}>0$.
The region ``in between'', i.e. $\hilbert\setminus(\mathcal{S}_{RC}\cup \mathcal{S}_{SC})$,
is a transition region where a mixture of both couplings is used.
The parameter $\delta$ occurs only for technical reasons and one
should think of $\delta$ being close to $0$.
\par\medskip
We now specify the technical realization of the coupling which follows 
\cite[Section 6]{eberle1}. For given
$(x,y)\in\hilbert\times\hilbert$,  we define linear operators
$R(x,y):\hilbert\rightarrow \hilbert$ and $S(x,y):\hilbert\rightarrow \hilbert$
by
\begin{eqnarray*}
	S(x,y) z &=& z^h \ +\  \operatorname{sc}(x,y)\;z^l \qquad\text{and}\\
	R(x,y) z &=& \operatorname{rc}(x,y)\;z^l.
\end{eqnarray*}
Here $\operatorname{sc}, \operatorname{rc}:\hilbert\oplus \hilbert
\rightarrow [0,1]$ are Lipschitz functions, satisfying for any $x,y\in \hilbert$,
\begin{eqnarray}\label{eq_cpcon}
	\operatorname{sc}^2(x,y)+\operatorname{rc}^2(x,y) &=& 1 \quad \text{and}\quad
	\operatorname{rc}(x,y) \ =\ 
						\begin{cases}
							1	& \text{if } (x-y)\in \mathcal{S}_{RC}.\\
							0	&  \text{if }  (x-y)\in \mathcal{S}_{SC}.
						\end{cases}
\end{eqnarray} \noindent
Regarding the existence of the above functions, we remark that it is enough to 
construct a suitable function 
$h:\reals_+\times\reals_+\rightarrow[0,1]$ such that
\begin{eqnarray*}
\operatorname{rc}(x,y)&=&h\left(\norm{x^h-y^h},\norm{x^l-y^l}\right) \quad\text{and}\quad \operatorname{sc}(x,y)\ =\ \sqrt{1-\operatorname{rc}^2(x,y)}
\end{eqnarray*}
satisfy the above
conditions.
This can be done using standard cut-off techniques.
Let now $\mathbb{W}^1$ and $\mathbb{W}^2$ be independent
$\G$-Wiener processes on $\hilbert$ and fix some arbitrary unit vector
$u\in \hilbert^l$.
Given starting points $(x_0,y_0)\in \hilbert\times \hilbert$ we define
$(X_t,Y_t)_{t\geq 0}$  as a strong solution of
\begin{align*}
	dX_t &= - X_t \,dt + b(X_t) \,dt +\sqrt{2}\,
	R(U_t) \,d\mathbb{W}^1_t + \sqrt{2}\,S(U_t)\, d\mathbb{W}^2_t,
\\ 	dY_t &= - Y_t \, dt + b(Y_t) \, dt + \sqrt{2}\G^{1/2}\, (I-2e_t
	\sProd{e_t}{\cdot})\, \G^{-1/2}\,R(U_t)\, d\mathbb{W}^1_t + \sqrt{2}
	S(U_t) \,d\mathbb{W}^2_t,
\end{align*}
on  $\hilbert \oplus \hilbert$, where $(X_0,Y_0)= (x_0,y_0)$, $U_t=(X_t,Y_t)$ and
\begin{align}\label{def_e_t^l}
	e_t=\begin{cases}
		\norm{\G^{-1/2}(X^l_t-Y^l_t)}^{-1}\G^{-1/2}(X^l_t-Y^l_t) & \text{ if
		}\norm{X_t^l-Y_t^l}>0,
		\\u 	& \text{ if
		}\norm{X_t^l-Y_t^l}=0.
	\end{cases}\end{align}
Notice that $\norm{X_t^l-Y_t^l}=0$ implies  $\operatorname{rc}(X_t,Y_t)=0$ and
thus the arbitrary value $u$ in \eqref{def_e_t^l} is not relevant for the
dynamic. The operator $\G^{-1/2}$ is well defined on the space $\hilbert^l$ due to Assumption
\ref{assNonDeg}. 
Furthermore, by assumption, the maps
$(x,y)\mapsto (b(x),b(y))$, $(x,y)\mapsto R (x,y)$ and $(x,y)\mapsto S(x,y)$ are Lipschitz on
$\hilbert\oplus\hilbert$. Observe that $(\mathbb{W}_t)$ defined by $\mathbb{W}_t=(\mathbb{W}_t^1,\mathbb{W}_t^2)$ is a
$\mathbb{G}$-Wiener process on $\hilbert\oplus \hilbert$ with
$\mathbb{G}(x,y)=(\G x,\G y)$.
Therefore, the above 
equation is a standard SDE with Lipschitz
coefficients on the Hilbert space $\hilbert\oplus \hilbert$. The existence of a
continuous, unique and non-explosive solution is well-known, see e.g. \cite[Theorem 3.3]{MR757771}.
Using the infinite-dimensional analog of Levy's characterization
of Brownian motion, see e.g. \cite[Theorem 4.4]{MR3236753}, and \eqref{eq_cpcon}
one can check that
\begin{eqnarray*} 
	t&\mapsto& \int_0^t R(U_s) \ d\mathbb{W}^1_s \ +\  \int_0^t S(U_s) \ d\mathbb{W}^2_s
	\quad \text{and}
	\\
	t&\mapsto& \int_0^t \G^{1/2}\, (I-2e_s
	\sProd{e_s}{\cdot})\, \G^{-1/2}\,R(U_s)\  d\mathbb{W}^1_s \ +\ 
	 \int_0^t S(U_s) \ d\mathbb{W}^2_s 
\end{eqnarray*}
are $\G$-Wiener processes on $\hilbert$ and hence $(X_t,Y_t)$ is indeed a
coupling.

\subsection{Proofs} \label{secProofs}
\begin{proof}[Proof of Proposition \ref{thmFirst}]
Fix initial values $x_0,y_0\in\hilbert$. We first show that \eqref{resEasyContraction} holds for Dirac measures $\mu=\delta_{x_0}$
and $\nu=\delta_{y_0}$. Let $(X_t,Y_t)$ be a synchronous coupling as defined in Section \ref{secCouplSynch}. 
In the following, all It{\^o}
differential (in)equalities hold almost surely for all $t\geq 0$ without further
mentioning.
\par\smallskip
Observe that the difference process $Z_t=X_t-Y_t$ satisfies the equation
\begin{eqnarray}\label{zeq}
	d Z_t & = & \left(\;-Z_t + b(X_t)-b(Y_t)\;\right) \ ds.
\end{eqnarray}
As before, we write  $Z^l_t$ and $Z^h_t$ for the orthogonal projections of $Z_t$ onto $\hilbert^l$ and $\hilbert^h$ respectively.
\begin{lemma}\label{thm1Lem1} 
The processes $(\norm{Z_t^l})$ and $(\norm{Z_t^h})$ satisfy the equations
	\begin{eqnarray}
 		d \norm{Z_t^l} &=& I_{Z_t^l\not= 0}\; \sProd{\frac{Z_t^l}{\norm{Z_t^l}}}{-Z_t+b(X_t)-b(Y_t)} \ dt,\label{thm1Lem1eq1}\\
 		d \norm{Z_t^h} &=& I_{Z_t^h\not= 0}\; \sProd{\frac{Z_t^h}{\norm{Z_t^h}}}{-Z_t+b(X_t)-b(Y_t)} \ dt.\label{thm1Lem1eq2}
	\end{eqnarray}
\end{lemma}

We proof Lemma \ref{thm1Lem1} further below and continue, assuming that it holds true. 
The coupling $(X_t,Y_t)$ yields an upper bound for the Kantorovich distance: 
\begin{eqnarray*}
		\mathcal{W}_{d_1}(\delta_{x_0} p_t,\delta_{y_0} p_t) &\leq& 
		E\left[\norm{Z_t}_\alpha  \right] \ = \  e^{-c\,t}\ E\left[ e^{c\,t}\norm{Z_t}_\alpha\ -\ \norm{Z_0}_\alpha\right] \ + \   e^{-c\,t}\ E\left[\norm{Z_0}_\alpha\right].
\end{eqnarray*} 
The product rule for semimartingales implies 
\begin{eqnarray}\label{t1a1}
	d(e^{c\,t}\norm{Z_t}_\alpha) &=& c\;e^{c\,t} \norm{Z_t}_\alpha \ dt \ + \  e^{c\,t} \ d\norm{Z_t}_\alpha.
\end{eqnarray}
Combining Lemma \ref{thm1Lem1} and \eqref{lemWeightedNorms1}, we conclude that
\begin{eqnarray}\label{t1a2} 
	d\norm{Z_t}_\alpha &\leq&  \left( \beta \norm{Z_t^l} - \alpha^{-1} \norm{Z_t^h}_\alpha\right) \ dt \ \leq \ - \,c \, \norm{Z_t}_\alpha \ dt. 
\end{eqnarray}
By \eqref{t1a1} and \eqref{t1a2}, $E\left[ e^{ct}\norm{Z_t}_\alpha-\norm{Z_0}_\alpha\right]\leq 0$ and therefore 
Proposition \ref{thmFirst} holds for Dirac measures. For the general case, let
$\mu,\nu\in\mathcal{P}$. With arguments similar to \cite[Theorem 4.8]{MR2459454} one can show that for any 
coupling $\gamma$ of $\mu$ and $\nu$, it holds
\begin{eqnarray*}
		\wDist_{d_1}( \mu p_t, \nu p_t) &\leq& \int \wDist_{d_1}(\delta_x p_t, \delta_y p_t)
		\ \gamma(dx\,dy)\ \leq \  e^{-c\, t}\; \int d_1(x,y) \ \gamma(dx\,dy).
\end{eqnarray*}
Taking the infimum over all couplings $\gamma$, we finish the proof of
Proposition \ref{thmFirst}.
\end{proof}

\begin{proof}[Proof of Lemma \ref{thm1Lem1}]
We argue pathwise. The chain rule combined with \eqref{zeq} yields
 	\begin{eqnarray}\label{lem1a1}
 		d \norm{Z_t^l}^2 &=& 2\; \sProd{Z_t^l}{-Z_t \ +\ b(X_t) - b(Y_t)} \ dt, 
 		\\
 		d \norm{Z_t^h}^2 &=& 2\; \sProd{Z_t^h}{-Z_t\ +\ b(X_t) - b(Y_t)} \ dt.
 	\end{eqnarray}
 	We introduce a $C^2$ approximation of the map $t\mapsto \sqrt{t}$.
 	Given $\epsilon>0$, we define
 	\begin{eqnarray}\label{squareRootApprox}
 			s(r)&=&\begin{cases}
 		  -({1}/{8})\;\epsilon^{-3/2}\; r^2\ +\ ({3}/{4})\; \epsilon^{-1/2}\; r\ +\ ({3}/{8})\;\epsilon^{1/2} & r<\epsilon\\
 		  \sqrt{r}	& r\geq \epsilon.
 	\end{cases} 
 	\end{eqnarray}
 	For any $r\in[0,\infty)$, $s(r)\rightarrow \sqrt{r}$ for $\epsilon\downarrow 0$. Let $r^l_t=\norm{Z_t^l}^2$. 
 	Using \eqref{lem1a1} and the chain rule, we see that for any $t\geq 0$,
 	\begin{eqnarray}\label{eqthm1appr1}
 		s(r^{l}_t)-s(r^{l}_0) &=& \int_0^t I_{r^{l}_u \geq \epsilon} \; 
 		\sProd{\frac{Z_u^l}{\norm{Z_u^l}}}{-Z_u\ +\ b(X_u)- b(Y_u)} \ du
 		\\&+& \int_0^t I_{0<r^{l}_u<\epsilon}\;s'(r^{l}_u)\;2\;
 		\sProd{Z_u^l}{-Z_u\ +\ b(X_u) - b(Y_u)} \ du.\label{eqthm1appr2}
 	\end{eqnarray} 
 	Observe that for $0<r^{l}_u<\epsilon$, 
 	\begin{eqnarray*}
 		\norm{\sProd{X_u^l-Y_u^l}{-(X_u-Y_u)+b(X_u)-b(Y_u)}} & \leq & \epsilon\ +\ \sqrt{\epsilon} \, \norm{b(X_u)-b(Y_u)}.
 	\end{eqnarray*}
 	Moreover, $\sup_{u\in[0,t]} \norm{b(X_u)-b(Y_u)}$ can be bounded by a constant,  
 	since $(X_u)$ and $(Y_u)$ are continuous and  $b$ is Lipschitz. Observe that 
 	$\sup_{0\leq r\leq\epsilon}\norm{s'(r)}\lesssim \epsilon^{-1/2}$. The Lebesgue dominated convergence
 	theorem yields that the integral \eqref{eqthm1appr2} vanishes in the limit as $\epsilon\downarrow 0$.
 	Arguing similarly for the integral on the r.h.s.\ of \eqref{eqthm1appr1}, we retrieve \eqref{thm1Lem1eq1}.
 	The proof of \eqref{thm1Lem1eq2} is analogous.
\end{proof}

\begin{proof}[Proof of Theorem \ref{thmSecond}]
We first define the function $f$ explicitly. The function is constructed using
a technique from \cite{eberle1,MR2843007}. Related constructions can be found in
\cite{MR1345035,MR1401516,MR2152380,2016arXivE}. For real numbers $a$ and $b$, we write $a\wedge b= \min\{a,b\}$.
\begin{align}\label{thm2funcf} 
	f(r) \ &=\ \int_0^{r} \phi(s\wedge R) \, g(s\wedge R) \ ds, &  \Phi(r)\ &=\ \int_0^{r} \,\phi(s\wedge R) \ ds,\\\nonumber
	\phi(r)\ &= \ \exp\left(-\frac{\beta }{8\lambda_\star}\,r^2\right), & \gamma^{-1}\ &=\ \int_0^{R} \phi(s)^{-1}\, \Phi(s) \ ds,\\\nonumber
	g(r) \ &=\ 1-\frac{\gamma}{2} \int_0^{r} \phi(s)^{-1} \, \Phi(s) \ ds. & &
\end{align}  
We summarize important properties. The derivative of the function
$f$ at $r\in(0,\infty)$ 
is given by the product $\phi(r\wedge R)\,g(r\wedge R)$. The functions $\phi$ and $g$ are strictly positive and
non-increasing on $(0,R)$ and thus $f$ is strictly increasing and 
concave. Notice that $g(R)=1/2$.
On the interval $[R,\infty)$ the function $f$ is linear with slope
$\phi(R)/2$. Moreover, for any $r\in(0,\infty)$,
\begin{eqnarray}\label{fineq}
	r&\leq& \phi(R)^{-1}  \Phi(r), \qquad \Phi(r)\ \leq \ r, \quad \text{and}\quad {\Phi(r)}/{2}\ \leq \  f(r) \ \leq \ \Phi(r),
\end{eqnarray}
which follows directly from the above definitions. Notice that $f(r)$ is twice
continuously differentiable at $r\in(0,R)$ and that it satisfies for such $r$ the (in)equality
\begin{eqnarray}\label{thm2pine}
	4 \,\lambda_\star\, f''(r) &=& - \beta\, f'(r)\, r - 2\, \lambda_\star\,\gamma \, \Phi(r) 
	\ \leq\  - \beta\, f'(r)\, r - 2\, \lambda_\star\,\gamma \,  f(r).
\end{eqnarray} 
We define the rate $c$ by 
\begin{eqnarray}\label{thm2c}
 	c&=&\min\left\{f'(R)\,(1-M),{f'(R)}/{(2\alpha)},2\,\lambda_\star\, \gamma\right\}. 
\end{eqnarray}
In order to see \eqref{clowerbound1}, observe that $f'(R)=\phi(R)/2$ and 
\begin{eqnarray*}
\gamma^{-1}&=&\int_0^{R} \phi(s)^{-1} \Phi(s) ds\ \leq\  \int_0^{R} \exp\left(\frac{\beta }{8 \lambda_\star}s^2\right) s\,ds\  =\  
 	 4\,\lambda_\star\, \frac{\exp\left(\frac{\beta }{8
 	 \lambda_\star}R^2\right)-1}{\beta}.
\end{eqnarray*}
\par\smallskip
Fix $(x_0,y_0)\in \hilbert\times \hilbert$. We argue that \eqref{eqThmScd} holds for Dirac measures 
$\mu=\delta_{x_0}$ and $\nu=\delta_{y_0}$.
Fix small $\delta>0$ and let $U_t=(X_t,Y_t)$ be the coupling with initial values $(x_0,y_0)$
defined in Section \ref{seccoupl3}. We use the notation $Z_t=X_t-Y_t$ and $r_t=\norm{Z_t}_\alpha$.
The coupling yields an upper bound for the Kantorovich distance:
\begin{align}\label{eqthm22}
 	\wDist_{d_2}(\delta_{x_0} p_t, \delta_{y_0} p_t) \ \leq\  E\left[f(r_t)\right] \ =\  e^{-c\,t} E\left[  e^{c\,t}f(r_t) - f(r_0)\right] +e^{-c\,t} E\left[f(r_0)\right].
 \end{align}
We now establish bounds on  $E\left[  e^{c\,t}f(r_t) - f(r_0)\right]$. All
It{\^o} differential (in)equalities hold almost surely for all $t\geq 0$ without
further mentioning.
\begin{lemma}\label{lem1}
 	The process $(r_t)$ satisfies
 	\begin{eqnarray*}
 			 dr_t &=& I_{Z_t^{l}\not= 0} \, \sProd{\frac{Z_t^{l}}{\norm{Z_t^{l}}}}{-Z_t+b(X_t)-b(Y_t)} \, dt \ +\  2\, \sqrt{2}\, \operatorname{rc}(U_t) \, \frac{\norm{Z_t^l}}{\norm{\G^{-1/2}Z_t^l}} \, dB_{t}
 			 \\&+&\alpha\, \,I_{Z_t^{h}\not= 0} \, \sProd{ \frac{Z_t^{h}}{\norm{Z_t^{h}}}}{-Z_t+b(X_t)-b(Y_t)} \, dt
 			,
 	\end{eqnarray*}
 	where $B_t= \int_0^t \sProd{\G^{-1/2} e_t}{d\mathbb{W}_{t}^{1,l}}$ is a
 	one-dimensional Brownian motion.
 \end{lemma}
 Observe that by \eqref{eq_cpcon} and \eqref{defSectors}, $Z_s^l=0$ implies $\operatorname{rc}(U_s)=0$.
 \begin{lemma}\label{lem2}
 	The process $\left(f(r_t)\right)$ satisfies
 	 \begin{eqnarray*}
 	df(r_t) &=& f'(r_t)\ dr_t \ +\  4\,
 	I_{r_t\not=R}\,f''(r_t)\, 
 	\operatorname{rc}(U_t)^2 \, {\norm{Z_t^l}^2}{\norm{\G^{-1/2}Z_t^l}^{-2}}\ dt.
 	 \end{eqnarray*}
 \end{lemma}
Assuming that Lemma \ref{lem2} holds true, we can apply the product rule for semimartingales to conclude
 \begin{eqnarray}\label{thm2eqa}
 d\left(e^{c\,t}f(r_t)\right) &=& c \, e^{c \, t} \, f(r_t)\ dt \ +\  e^{c \,t} \ df(r_t).
 \end{eqnarray}
\begin{lemma}\label{lem3}
 	There is a function $h:\reals_+\rightarrow \reals_+$ with $\lim_{r\downarrow
 	0}h(r)=0$ such that
 	\begin{eqnarray}\label{lem3eq2}
 	df(r_t) &\leq& - c \, f(r_t)\ dt \  +\   h(\delta)\ dt \ +\  f'(r_t)\ dM^1_t,
 	\end{eqnarray}
 	where $M^1_t=\int_0^t  2\, \sqrt{2}\, \operatorname{rc}(U_t) \, {\norm{Z_t^l}}\,{\norm{\G^{-1/2}Z_t^l}^{-1}} \, dB_{t}$.
\end{lemma}
Notice that $M^1_t$ is a martingale and $f'\leq 1$. By Lemma \ref{lem3}, \eqref{eqthm22} and \eqref{thm2eqa},  
  \begin{align*}
 	\wDist_{d_2}(\delta_{x_0} p_t, \delta_{y_0} p_t) &\leq h(\delta)/c +e^{-c\,t}\,
 	\wDist_{d_2}(\delta_{x_0}, \delta_{y_0}).
 \end{align*}
Passing to the limit $\delta\rightarrow 0$, we see that \eqref{eqThmScd} holds for Dirac measures.
The general case can be concluded with the same argument used at the end of the 
 proof of Proposition \ref{thmFirst}.
\end{proof}

\begin{proof}[Proof of Lemma \ref{lem1}]
 We first consider the projection of $Z_t$ onto $\hilbert^h$. From the
  definition of the coupling in Section \ref{seccoupl3}, we see that
\begin{eqnarray*}
	dZ_t^{h} &=&\left(\;-Z_t^{h}\ +\ b^h(X_t)-b^h(Y_t)\; \right)\ dt.
\end{eqnarray*} 
Using the same approximation argument as in the proof Lemma \ref{thm1Lem1}, we conclude
\begin{eqnarray*}
d\norm{Z_t^{h}} &=&   I_{Z_t^{h}\not= 0}\, {\norm{Z_t^{h}}}^{-1}\,
 \sProd{{Z_t^{h}}}{-Z_t \ + \ b(X_t)-b(Y_t)} \ dt.
\end{eqnarray*} 
Using the definition of the coupling in Section \ref{seccoupl3}, we see that 
the projection of $Z_t$ onto $\hilbert^l$ satisfies
\begin{eqnarray*}
	dZ_t^{l} &=& \left(-Z_t^{l}+b^l(X_t)-b^l(Y_t)\right)\ dt \ +\ 
2\,\sqrt{2}\, \operatorname{rc}(U_t)\, \G^{1/2}\, e_t
 \,\sProd{\G^{-1/2} e_t}{d\mathbb{W}_{t}^{1,l}},
\end{eqnarray*} 
Notice that $B_t= \int_0^t \sProd{\G^{-1/2} e_s}{d\mathbb{W}_{s}^{1,l}}$ is a
one-dimensional Brownian motion, which follows from Levy's characterization of Brownian motion.
A Hilbert space version of It\^o's formula, see e.g. \cite[Theorem 2.9]{MR2560625}, allows to conclude
 \begin{eqnarray*}
 	 d\norm{Z_t^{l}}^2 &=&  2\,\sProd{Z_t^{l}}{-Z_t\ +\ b(X_t)-b(Y_t)} dt + 8 \,\operatorname{rc}(U_t)^2\, \norm{\G^{1/2} e_t}^2 dt 
 	 \\&+& 4\,\sqrt{2}\, \operatorname{rc}(U_t) \,\sProd{Z_t^{l}}{\G^{1/2} e_t} \, dB_t.
 \end{eqnarray*}
 Given $\epsilon>0$, let $s(t)$ be the $C^2$ approximation of $t\mapsto \sqrt{t}$ defined in \eqref{squareRootApprox}.
 It{\^o}'s formula shows 
 \begin{eqnarray} \label{x1}
 	s\left(\norm{Z_t^{l}}^2\right)-s\left(\norm{Z_0^{l}}^2\right) &=& \int_0^t s'(\norm{Z_v^{l}}^2) \,2\,\sProd{Z_v^{l}}{-Z_v\ +\ b(X_v)-b(Y_v)} dv\\
 	&+& \int_0^t s'(\norm{Z_v^{l}}^2) \, 8 \,\operatorname{rc}(U_v)^2\, \norm{\G^{1/2} e_v}^2 dv\nonumber\\
 	&+& \int_0^t s''(\norm{Z_v^{l}}^2) \, 16 \operatorname{rc}(U_v)^2 \,(\sProd{Z_v^{l}}{\G^{1/2} e_v})^2 dv\nonumber \\
 	&+& \int_0^t s'(\norm{Z_v^{l}}^2) \,4\,\sqrt{2}\, \operatorname{rc}(U_v) \,\sProd{Z_v^{l}}{\G^{1/2} e_v} \, dB_v.\nonumber
 \end{eqnarray} 
We now pass to the limit $\epsilon\downarrow 0$. 
The integral on the r.h.s.\ of \eqref{x1} converges to   
 	\begin{align*}
 		\int_0^{t}  I_{Z_u^{l}\not = 0} \,\norm{Z_u^{l}}^{-1}\,\sProd{Z_u^{l}}{-Z_u+b(X_u)-b(Y_u)}\, du,
 	\end{align*}
 	which can be argued similarly as in the proof of Lemma \ref{thm1Lem1}. Regarding the limits of the
 	remaining integrals, notice that by \eqref{defSectors} and \eqref{eq_cpcon},  
 	\begin{eqnarray}\label{lem3q2}
 		\norm{Z_t^{l}}&<& \frac{\delta}{4}\; \min\left\{1, \frac{1}{4\;\alpha\;(\beta+1)}\right\}
 	\end{eqnarray}
  	implies $\operatorname{rc}(U_t)=0$.
 	Indeed, if $\norm{Z_t^{h}}\leq \delta/(4\alpha)$ and \eqref{lem3q2} holds, then $\norm{Z_t}_\alpha<\delta/2$
 	and thus $Z_t\in \mathcal{S}_{\operatorname{SC}}$. If $\norm{Z_t^{h}}>\delta/(4\alpha)$ and \eqref{lem3q2} holds,  
 	then 
 	\begin{eqnarray*}
 		4\;(\beta+1) \norm{Z_t^{l}} &<& \delta/(4\alpha) \ < \ \norm{Z_t^{h}}
 	\end{eqnarray*}
 	and thus again $Z_t\in \mathcal{S}_{\operatorname{SC}}$. 
 	On the other hand, $s(t)=\sqrt{t}$ for $t\geq \epsilon$,
 	which concludes the lemma. 
 \end{proof}
 
\begin{proof}[Proof of Lemma \ref{lem2}]
The function $f$ can be continued to a concave function on $\reals$ by setting $f(x)=x$ for $x<0$. 
 The generalized
 It\^o formula for concave functions, see e.g.\ \cite[Thm.
 22.5]{MR1876169}, implies that $(f(r_t))$ satisfies the equation 
 \begin{eqnarray} \label{eqthm21}
 	f(r_t)-f(r_0) &=& \int_0^{t} f'_{-}(r_s)\ dr_s \ + \ \frac{1}{2}
 	\int_{-\infty}^\infty L_t^x \ \mu_f(dx),
 \end{eqnarray}
 where $f'_{-}$ denotes the left-derivative of $f$, $\mu_f$ is the
 signed measure induced by $f'_{-}$, i.e.\
 $\mu_f[x,y)=f'_{-}(y)-f'_{-}(x)$ for $x\leq y$, and $L_t^x$ denotes the
 right-continuous local time of $(r_t)$. A further consequence of the generalized 
 It\^o formula 
 is that, outside of a fixed null set, we retrieve for any measurable and non-negative function $v:\reals\rightarrow \reals_+$ 
 the equality
 \begin{eqnarray}
 \label{eqthm2222} \int_\reals L_t^x\; v(x)\ dx &=& \int_0^t v(r_s)\ 
 d[r]_s \qquad \forall t\geq 0,
 \end{eqnarray} see e.g.\ \cite[Thm. 22.5]{MR1876169}.
 Since $f'$ exists everywhere and is continuous, we have  $\mu_f[\{R\}]=0$.
 Observe that $f$ is twice continuously differentiable except at the point $R$.
 Hence by \eqref{eqthm21}, \eqref{eqthm2222} and Lemma \ref{lem1}, we can conclude that $(f(r_t))$ satisfies the equations 
 \begin{eqnarray*}
 	f(r_t)-f(r_0)&=&\int_0^{t} f'(r_s)\ dr_s \ + \  \frac{1}{2} \int_{-\infty}^\infty I_{x\not=R} \,L_t^x\ \mu_f(dx),\\
 	\int_{-\infty}^\infty I_{x\not=R} \,L_t^x\, \mu_f(dx)&=&\int_{-\infty}^\infty I_{x\not=R}\, L_t^x \,f''(x) \,dx \ = \ 
 	\int_{0}^t I_{r_s\not=R} \,f''(r_s) \,d[r]_s
 \\&=&  \int_{0}^{t}
 I_{r_s\not=R} \,f''(r_s)\, 8\, \operatorname{rc}(U_s)^2\,
 \frac{\norm{Z_s^l}^2}{\norm{\G^{-1/2}Z_s^l}^2}\ ds.
 \end{eqnarray*}
 \end{proof}
 
 \begin{proof}[Proof of Lemma \ref{lem3}] 
Let 
 \begin{eqnarray*}
 w(U_t)&=&I_{Z_t^{l}\not= 0} \,{\norm{Z_t^{l}}}^{-1}\, \sProd{{Z_t^{l}}}{-Z_t+b(X_t)-b(Y_t)}
 			 \\&+& \alpha\,I_{Z_t^{h}\not= 0} \, {\norm{Z_t^{h}}}^{-1}\,\sProd{ {Z_t^{h}}}{-Z_t+b(X_t)-b(Y_t)}.
 \end{eqnarray*}
 Combining Lemma \ref{lem1} and \ref{lem2}, we conclude
 \begin{eqnarray}\label{ref1} 
 	df(r_t) &=& \left(f'(r_t)\,  w(U_t)\, + 4 \, I_{r_t\not=R}\,f''(r_t)\, 
 	\operatorname{rc}(U_t)^2 \, \frac{\norm{Z_t^l}^2}{\norm{\G^{-1/2}Z_t^l}^{2}}\,\right) dt + dM_t^2. 
 \end{eqnarray} 
 with $M_t^2=\int_0^t f'(r_t)\ dM_t^1$.
 Notice that for any $t\geq 0$, 
 \begin{eqnarray}\label{Gest}
 	 {\norm{\G^{-1/2}Z_t^l}} &\leq& \lambda_\star^{-1/2} \; \norm{Z_t^l}.
 \end{eqnarray}
 Moreover, Lemma \ref{lemWeightedNorms} implies that
\begin{eqnarray}\label{w1}
	w(U_t) &\leq& -r_t \ +\  \norm{b(X_t)-b(Y_t)}_\alpha \ \leq \ - \norm{Z_t^h} + \beta \norm{Z_t^l}.
\end{eqnarray}
Recall that $f$ is concave and non-decreasing. By \eqref{ref1}, \eqref{Gest} and \eqref{w1},
  \begin{eqnarray}\label{lem3q1} 
 	df(r_t) &\leq & f'(r_t)\, \left(-r_t \ +\  \norm{b(X_t)-b(Y_t)}_\alpha\right)\ dt
 	\\\nonumber &+& 4\,\lambda_\star \, I_{r_t\not=R}\,f''(r_t)\, \operatorname{rc}(U_t)^2 \ dt \ +\  dM^2_t. 
 \end{eqnarray} 
If $r_t\geq R$, then Assumption \ref{assLocalNonContractive} and \eqref{fineq} imply
 \begin{eqnarray}\label{ref2}
 - r_t \ +\  \norm{b(X_t)-b(Y_t)}_\alpha  &\leq& - (1-M) \,r_t \ \leq \ -(1-M) \,f(r_t).
 \end{eqnarray}
 If $Z_t\not\in\mathcal{S}_{\operatorname{RC}}$ and $r_t\geq \delta$, then by \eqref{defSectors} 
 and Lemma \ref{lemWeightedNorms} 
\begin{eqnarray} \label{ref3} 
	- r_t \ +\  \norm{b(X_t)-b(Y_t)}_\alpha  &\leq&  -{1}/{(2\alpha)} \, r_t \ \leq \  -{1}/{(2\alpha)} \, f(r_t).
\end{eqnarray}
If $Z_t\in\mathcal{S}_{\operatorname{RC}}$ and $\delta\leq r_t<R$, we argue as follows:
By definition we have $\operatorname{rc}(U_t)=1$. Lemma \ref{lemWeightedNorms}
implies the bound
\begin{eqnarray} \label{ref4}
	- r_t \ +\  \norm{b(X_t)-b(Y_t)}_\alpha & \leq & \beta \, r_t.
\end{eqnarray}
Observe that for $r\in (0,R)$, inequality \eqref{thm2pine} holds true and therefore
\begin{eqnarray}\label{lem3q3}
	f'(r_t)\,  \beta \, r_t\, + 4\,\lambda_\star \,f''(r_t)&\leq& - 2\,\lambda_\star\,\gamma\, f(r_t).
\end{eqnarray}
Recall \eqref{thm2funcf} and \eqref{fineq} to see that if $r_t\leq \delta$ holds, then we can estimate
\begin{eqnarray}\label{lem3q21}
f'(r_t)\,  \beta \, r_t\, &\leq&  \beta\; \delta \qquad\text{and}\qquad f(r_t)\ \leq\  r_t \ \leq \  \delta.
\end{eqnarray}
Combining \eqref{lem3q1}, \eqref{ref2}, \eqref{ref3}, \eqref{ref4}, \eqref{lem3q3}, \eqref{lem3q21} and \eqref{thm2c},
we conclude
\begin{eqnarray} 
 	df(r_t) &\leq & -c\, f(r_t) \ dt \ +\ \left( c + \beta \right)\delta \ dt \ + \  dM^2_t. 
 \end{eqnarray} 
 The claim follows by setting $h(\delta)\ =\ \left( c \ +\ \beta \right)\; \delta$. 
\end{proof}

\begin{proof}[Proof of Corollary \ref{corTHM21}]
By \eqref{fineq}, \eqref{eqThmScd} and \eqref{normineq}, we conclude for any $x,y\in\hilbert$,
\begin{eqnarray*}
	\wDist^1(\delta_x p_t,\delta_y p_t) &\leq& 2 \; \phi(R)^{-1} \wDist_{d_2}(\delta_x p_t,\delta_y p_t) \ \leq \ 4\,\alpha\, \phi(R)^{-1} \;e^{-c\; t} \ \wDist^1(\delta_x,\delta_y).
\end{eqnarray*}
 		The fact that the Markov kernels $(p_t)$ admit a unique invariant measure $\pi$ satisfying $\pi p_t=\pi$ for any $t\geq 0$ now follows
 		by standard arguments, see e.g.\ \cite[Corollary 2.5]{eberle1}.
 \end{proof}
 
 \begin{proof}[Proof of Corollary \ref{corTHM23}] 
 The proof follows \cite[Section 4]{eberle1}. Let $(X_t)$ be a solution of \eqref{eqMain} with $X_0=x$.
 Assumption \ref{assLocalNonContractive}
 implies that the first moments of $X_t$ are uniformly bounded in time,
  	 i.e.\ 
	 $\sup_{t\geq 0}E_x[\norm{X_t}]<\infty$. 
  	 In particular, for any $x\in\hilbert$, $t\geq 0$ and 
	 any Lipschitz function $g$, $\int g(y) \, p_t(x,dy)<\infty$. Fix $x,y\in\hilbert$ and let 
	 $(X_t,Y_t)$
	 be \emph{any} coupling  of $\delta_x p_t$ and $\delta_y p_t$. It follows
	 \begin{align*} 
	 	\norm{(p_tg)(x)-(p_tg)(y)}\leq E[\norm{g(X_t)-g(Y_t)}]\leq
	 	\dnorm{g}_{\operatorname{Lip}(d_2)} E[d_2(X_t,Y_t)],
	 \end{align*}
	 and hence by \eqref{eqThmScd}, \eqref{normineq} and \eqref{fineq},
	 \begin{align*}
	 	\norm{(p_tg)(x)-(p_tg)(y)} \leq \dnorm{g}_{\operatorname{Lip}(d_2)} e^{-c\,t}
	 	f\left(\norm{x-y}_\alpha\right)\leq \sqrt{2}\,\alpha\,
	 	\dnorm{g}_{\operatorname{Lip}(d_2)} \, e^{-c\,t}\, \norm{x-y}.
	 \end{align*}
 \end{proof}
 
\begin{proof}[Proof of Theorem \ref{thmThird}]
	The proof is close to the proof of Theorem \ref{thmSecond}. We use again the coupling from Section
	\ref{seccoupl3}, but use a slightly different function $f$.
	\begin{align}\label{thm3funcf}
	f(r) \ &=\ \int_0^{r\wedge R} \phi(s)\, g(s) \ ds & \Phi(r)\ &=\ \int_0^{r\wedge R} \phi(s) \ ds\\\nonumber
	\phi(r) \ &= \ \exp\left(-\frac{\constB}{8\lambda_\star}
	r^2-2\theta\frac{\lambda^\star}{\lambda_\star}r\right) & 
	\gamma^{-1} \ &= \ \int_0^{R} \Phi(s)\, \phi(s)^{-1} \ ds
	\\\nonumber
	g(r)\ &=\ 1-\frac{\gamma}{2}\int_0^{r\wedge R}
	\Phi(s)\, \phi(s)^{-1}\ ds
\end{align}

We highlight the differences to the situation in Theorem \ref{thmSecond}.
This time, $f$ is constant on $[R,\infty)$ and it is not differentiable at the point $R$.
Nevertheless, it is concave and the left-derivative  $f'_{-}$ exists everywhere. 
Observe that the inequalities \eqref{fineq} still hold true on the interval $[0,R]$.
Moreover, the function $f$ is twice
continuously differentiable on $(0,R)$ and satisfies there
the (in)equality
\begin{eqnarray}\label{thm3pine}
4\,\lambda_\star\, f''(r)&=&-f'(r)\,\left(\beta \, r + 8\,\theta\,\lambda^\star\right) -
2\, \lambda_\star\, \gamma\, \Phi(r) \\&\leq& -f'(r)\,\left(\beta \,r +
8\,\theta\,\lambda^\star\,\right) - 2\,\lambda_\star \,\gamma \,f(r).\nonumber
\end{eqnarray}
The contraction rate $c$ in \eqref{eqThm3rd} and the constant $\epsilon$ in \eqref{d3distance} are given by
\begin{eqnarray}\label{thm3rdContrRate}
	c &=& \min\left\{\lambda_\star
 			\gamma,
		 \frac{\phi(R)}{8\;\alpha},\frac{\eta}{2}\right\} \quad\text{and}\quad
	2\;C\; \epsilon \ =\ 
	\min\left\{\lambda_\star \gamma, \frac{\phi(R)}{8\;\alpha}\right\}\geq c.
\end{eqnarray}
The lower bound \eqref{clowerbound2} can be derived similarly as in the proof of
Theorem \ref{thmSecond}.
\par\medskip
Fix small $\delta>0$, initial conditions $(x_0,y_0)\in\hilbert\times\hilbert$ and let $U_t=(X_t,Y_t)$ be the
coupling defined in Section \ref{seccoupl3}. We use the notation 
\begin{align*}
Z_t\ &=\ X_t-Y_t, & r_t\ &=\ \norm{Z_t}_\alpha,  \\
G(x,y) \ &=\   1\ +\ \epsilon\, V(x)+\epsilon\, V(y) & Q_t \ &= \ f(r_t)\;G(X_t,Y_t).
\end{align*}
The coupling yields an upper bound for the Kantorovich distance:
 \begin{eqnarray}\label{aimaim}
 	\wDist_{d_3}(\delta_{x_0} p_t, \delta_{y_0} p_t) &\leq& E\left[Q_t\right] 
 	\ =\  e^{-c\,t} E\left[  e^{c\,t}Q_t - Q_0\right]\ +\ e^{-c\,t} E\left[Q_0\right].
 \end{eqnarray} 
We now estimate $E\left[  e^{c\,t}Q_t - Q_0\right]$ and proceed similarly to the proof of Theorem \ref{thmSecond}.
Observe that Lemma \ref{lem1} still holds true, since we use the same coupling as in the proof of Theorem \ref{thmSecond}.
\begin{lemma} \label{thm3lem1}
 	The process $\left(f(r_t)\right)$ satisfies
 	 \begin{eqnarray*}
 	 df(r_t) &=& f'_{-}(r_t) \ dr_t \ +\ \frac{1}{2}
 	\int_{-\infty}^\infty L_t^x \ \mu_f(dx)\\
 	&\leq& f'_{-}(r_t) \ dr_t \ +\  4\,
 I_{r_t \not=R} \, f''(r_t) \,\operatorname{rc}(U_t)^2\,
 {\norm{Z_t^l}^2}{\norm{\G^{-1/2}Z_t^l}^{-2}}\, dt.
 \end{eqnarray*} 
 \end{lemma}	
 The notation $\mu_f$ and $L_t^x$ is defined in the proof of Lemma \ref{lem2}.
\begin{lemma}\label{thm3lem2}
     The process $(G(X_t,Y_t))$ satisfies
 \begin{eqnarray}\label{k1}
 	dG(X_t,Y_t) &=& \epsilon\,\left(\; \generator{L} V(X_t) + \generator{L} V(Y_t)\; \right) \,dt \ +\  dM^3_t,
 	\end{eqnarray}
 	where $(M_t^3)$ is a local martingale given by
 	\begin{eqnarray*}
 	dM_t^3 &=& 
 		\sqrt{2}\,\epsilon\, \sProd{\mathcal{D}
 	V(X_t)+\mathcal{D}
 	V(Y_t)}{d\mathbb{W}^{2,h}_{t}} 
 	\\&+& \sqrt{2}\,\epsilon\, \operatorname{sc}(U_t)
 	\sProd{\mathcal{D} V(X_t)+\mathcal{D}
 	V(Y_t)}{d\mathbb{W}^{2,l}_{t}}
 	\\&+&\sqrt{2}\,\epsilon\, \operatorname{rc}(U_t) \sProd{\mathcal{D}
 	V(X_t)+\mathcal{D}
 	V(Y_t)}{d\mathbb{W}^{1,l}_{t}}\\
 	&-&2\sqrt{2}\,\epsilon\, \operatorname{rc}(U_t)  \sProd{\mathcal{D}
 	V(Y_t)}{\G^{1/2}e_t}
	\sProd{ \G^{-1/2}e_t}{d\mathbb{W}^{1,l}_{t}}.
 \end{eqnarray*} 
 \end{lemma} 
The product rule for semimartingales implies 
 \begin{eqnarray}\label{thm2eqa2}
 d\left(e^{c\,t}Q_t\right) &=& c \, e^{c \, t} \, Q_t\, dt \ +\  e^{c \,t} \, dQ_t.
 \end{eqnarray}

\begin{lemma}\label{thm3lem3}
 	There is $h:\reals_+\rightarrow \reals_+$ with $\lim_{r\downarrow
 	0}h(r)=0$ such that
 	\begin{eqnarray}\label{q1}
 		dQ_t &\leq& - c\, Q_t \ dt \ +\  (1 + \epsilon\, V(X_t)  + \epsilon\, V(Y_t))\, h(\delta)\ dt \ +\ dM^4_t,
 	\end{eqnarray}
 	where $M^4_t=\int_0^t f(r_s) \ dM^3_s+\int_0^t G(X_s,Y_s)\; f'_{-}(r_s)\ dM^1_s$ is a local martingale.
 \end{lemma}
 The martingale $(M_t^1)$ is defined in Lemma \ref{lem3}.

 \begin{lemma}\label{finallemma}
 For any $t\geq 0$, there is $K_t\in(0,\infty)$, not depending on $\delta$, such that
 \begin{eqnarray*}
 		E\left[  e^{c\,t}Q_t- Q_0\right] &\leq&  K_t\, h(\delta).
 \end{eqnarray*}
 \end{lemma}
 Combining Lemma \ref{finallemma} with \eqref{aimaim} yields
  \begin{eqnarray*}
 	\wDist_{d_3}(\delta_{x_0} p_t, \delta_{y_0} p_t) &\leq& K_t\, h(\delta) \ +\ e^{-c\,t}\
 	\wDist_{d_3}(\delta_{x_0}, \delta_{y_0}).
 \end{eqnarray*}
 Passing to the limit $\delta\rightarrow 0$, we see that \eqref{eqThm3rd} holds for Dirac measures.
 The general case can be concluded with the same argument used at the end of the 
 proof of Proposition \ref{thmFirst}.
\end{proof}
 
 \begin{proof}[Proof of Lemma \ref{thm3lem1}] 
  The proof is analogous to the proof of Lemma \ref{lem2}, except that now $f$ is not continuously differentiable everywhere. 
  In particular, $f_{-}'$ has a discontinuity at the point $R$ and therefore we do not have  
  $\mu_f[\{R\}]=0$. Nevertheless, since $f$ is concave we know that $\mu_f[\{R\}]<0$. 
 \end{proof} 
 
 \begin{proof}[Proof of Lemma \ref{thm3lem2}]
 	The assumptions imposed on $V$ allow to apply It{\^o}'s formula in a Hilbert space setting, see e.g.\ \cite[Theorem 2.9]{MR2560625}. 
 	Recalling the definition of the coupling from Section \ref{seccoupl3}, we see that \eqref{k1} holds true. 
 \end{proof}
 
 \begin{proof}[Proof of Lemma \ref{thm3lem3}] 
The product rule for semimartingales implies that $(Q_t)$ satisfies
	\begin{eqnarray}\label{huba}
		dQ_t &=& G(X_t,Y_t)\ df(r_t) \ +\ f(r_t) \
		dG(X_t,Y_t) \ +\  d\left[f(r_\cdot),G(X_\cdot,Y_\cdot)\right]_t,
	\end{eqnarray}
	where $[\cdot,\cdot]$ denotes the quadratic variation. By Lemma \ref{thm3lem1}, 
	\eqref{Gest} and \eqref{w1}, 
	\begin{eqnarray} \label{thm3lem3eq1}
		G(X_t,Y_t) \ df(r_t) &\leq& G(X_t,Y_t) \,f'_{-}(r_t)\,
		\left(-\norm{Z_t^h}+\beta \norm{Z_t^l}\right) \,dt \\&+& 
		G(X_t,Y_t)\, 4\, \lambda_\star \, I_{r_t\not=R} \, f''(r_t)\,
		\operatorname{rc}(U_t)^2 dt + dM^5_t,\nonumber
	\end{eqnarray}
	where $M_t^5=\int_0^t G(X_s,Y_s) f'_{-}(r_s)\ dM^1_s$ is a local martingale.
	\par\smallskip
	Lemma \ref{thm3lem2} and Assumption \ref{assLypDri} imply that
	\begin{eqnarray}
		 f(r_t)\  dG(X_t,Y_t) &\leq& f(r_t)\,\epsilon\, \left(2\,C-\eta\,\left(V(X_t)+V(Y_t)\right)\right) dt \ +\  dM^6_t,\label{thm3lem3eq2}
	\end{eqnarray} 
	with $M_t^6=\int_0^t f(r_t)\ dM_t^3$. Using Lemma \ref{lem1}, \ref{thm3lem1}
and \ref{thm3lem2}, we establish the bound
	\begin{align*}
		[f(r_\cdot),G(X_\cdot,Y_\cdot)]_t 
	&=4\epsilon \int_0^t f'_{-}(r_s) \operatorname{rc}(U_s)^2
	\frac{\norm{Z_s^l}}{\norm{\G^{-1/2}Z_s^l}^2} \sProd{\mathcal{D}
 	V(X_s)-\mathcal{D}V(Y_s)}{Z_s^l} ds\\
 	&\leq 4\;\epsilon\;\lambda^\star \, \int_0^{t} f'_{-}(r_s) \,
 	\operatorname{rc}(U_s)^2 (\; \norm{\mathcal{D} V(X_s)}+\norm{\mathcal{D}V(Y_s)}\;) \, 
 	 ds,
 	 \end{align*}
 	 where $\lambda^\star$ is the largest eigenvalue of $\G$ on $\hilbert^l$.	
	 Assumption \ref{assLypDri} implies
 	 \begin{eqnarray} \label{qvar}
 	  [f(r_\cdot),G(X_\cdot,Y_\cdot)]_t &\leq& 8\, \theta\, \lambda^\star\,
 	  \int_0^{t} f'_{-}(r_s) \,\operatorname{rc}(U_s)^2 \,G(X_s,Y_s) \ ds.
	\end{eqnarray}
	Combining \eqref{huba}, \eqref{thm3lem3eq1}, \eqref{thm3lem3eq2} and \eqref{qvar}, 
	we conclude that
	\begin{eqnarray*}
		dQ_t &\leq& G(X_t,Y_t) f'_{-}(r_t)\,
		\left(-\norm{Z_t^h}+\beta \norm{Z_t^l} 
		+ 8\,\theta\,\lambda^\star \operatorname{rc}(U_t)^2 \right) \ dt 
		\\&+&
		G(X_t,Y_t)\, 4\, \lambda_\star \, I_{r_t\not=R} \, f''(r_t)\,
		\operatorname{rc}(U_t)^2 \ dt  
		\\&+&f(r_t)\,\epsilon\, \left(2\,C-\eta\,\left(V(X_t)+V(Y_t)\right)\right) \, dt
		\ +\ dM^5_t \ +\ dM^6_t.
	\end{eqnarray*}
	We are now in a position to argue \eqref{q1} and do a pathwise case distinction: 
	\par\medskip
	If $r_t>R$, then $(X_t,Y_t)\not\in S$ by \eqref{defS}.
By \eqref{genIneq} and \eqref{thm3rdContrRate},
\begin{eqnarray*}
	f(r_t)\,\epsilon\, \left(2\,C-\eta\,\left(V(X_t)+V(Y_t)\right)\right) &\leq& f(r_t)\,\left(-2\,C\, \epsilon -\eta/2 \left(\epsilon\, V(X_t)+\epsilon\,V(Y_t)\right) \right)
	\\&\leq& - c \ f(r_t)\,G(X_t,Y_t) \ =\  -c \ Q_t. 
\end{eqnarray*}
 	Moreover, $f$ is constant on $(R,\infty)$ and thus $f'(r_t)=f''(r_t)=0$.
 \par\medskip
 Now assume that $\delta\leq r_t\leq R$ and $Z_t\in
 	 	\mathcal{S}_{\operatorname{RC}}$. By \eqref{eq_cpcon}, we have that $\operatorname{rc}(U_t)=1$.
Notice that equality \eqref{eqthm2222}  implies for any fixed $t\geq 0$,
 	 	\begin{align}\label{thm3lem1eq2}
 	 		\mathcal{\lambda_{\operatorname{Leb}}}\left(\{0\leq s\leq t:
 	 		r_s=R \text{ and } \operatorname{rc}(U_s)>0\}\right)=0,
 	 	\end{align}
 	 	i.e.\ the Lebesgue measure of the time $(r_s)$ spends at the point $R$ up to time $t$, while $\operatorname{rc}(U_s)>0$ is almost surely zero.	This allows us to neglect the case $r_t=R$.  Moreover, $f$ is twice continuously
 	 	differentiable on $(0,R)$ and fulfils inequality \eqref{thm3pine}. We conclude that
 	 	for $\delta\leq r_t<R$ with $Z_t\in
 	 	\mathcal{S}_{\operatorname{RC}}$,
 	 	\begin{eqnarray*}
 	 	&&G(X_t,Y_t) \,f'(r_t)\,
		\left(\beta \norm{Z_t^l} \ + \  8\,\theta\,\lambda^\star \right) +
		G(X_t,Y_t)\, 4\, \lambda_\star \, f''(r_t) + f(r_t)\,\epsilon\, 2\,C
		\\&\leq& - 2\, \lambda_\star \, \gamma \, G(X_t,Y_t) \, f(r_t) \ +\  f(r_t)\,\epsilon\, 2\,C
		\ \leq\  - c \, Q_t,
		\end{eqnarray*}
		where we used  \eqref{thm3rdContrRate} and $G\geq 1$.
 \par\medskip
 If  $\delta\leq r_t\leq R$ and $Z_t\not\in \mathcal{S}_{\operatorname{RC}}$, then 
 by \eqref{defSectors}, $2\left(\beta+1\right)\norm{Z_t^l}\leq \norm{Z_t^h}$, but 
 we do not necessarily have 
 	 	$\operatorname{rc}(U_t)=1$. Nevertheless, \eqref{thm3lem1eq2} is still true
 	 	and \eqref{thm3pine} implies
 	 	\begin{eqnarray}\label{eqq1}
 	 				f'_{-}(r_t) \, 8 \, \theta \,\lambda^\star
		 \operatorname{rc}(U_t)^2 \ +\  4 \,\lambda_\star\,
		 f''(r_t) \,\operatorname{rc}(U_t)^2&\leq& 0 \qquad\text{ for } 0<r_t<R.
 	 	\end{eqnarray} 
 	 	Lemma \ref{lemWeightedNorms} shows that
 	 	\begin{eqnarray}
 	 		-\norm{Z_t^h}\ +\ \beta \norm{Z_t^l} &\leq&  -{1}/{(2\alpha)}\;
r_t \ \leq\  -{1}/{(2\alpha)}\;f(r_t)
 	 	\end{eqnarray}
	and thus
	\begin{eqnarray*}
		&&G(X_t,Y_t) \,f'_{-}(r_t)\,
		\left(-\norm{Z_t^h}+\beta \norm{Z_t^l} \right)\ +\ f(r_t)\,\epsilon\,2\, C
		\\& \leq& -\phi(R)/(4\alpha)\, Q_t \ +\    f(r_t)\,\epsilon\,2\, C \leq  -c\, Q_t,
	\end{eqnarray*}
	where we used \eqref{thm3rdContrRate} and the fact that $f'$ is non-negative and decreasing on $(0,R)$ with $f'_{-}(R)=\phi(R)/2$.
	\par\medskip  
	Now assume $r_t \leq \delta$. Similarly to the last case, \eqref{eqq1} holds true.  
 	 	Since $f'_{-}\leq 1$  and $f(r)\leq r$, we can estimate
\begin{eqnarray*}
G(X_t,Y_t) \,f'_{-}(r_t)\, \beta \norm{Z_t^l} \ +\  f(r_t)\,\epsilon\, 2\,C &\leq& G(X_t,Y_t) \left(\beta+2\,C\,\epsilon\right) \delta
\end{eqnarray*}
We conclude the lemma setting $h(\delta)=\left(c+\beta+2\,C\,\epsilon\right)\delta$.  	 	
\end{proof}

\begin{proof}[Proof of Lemma \ref{finallemma}]
	We introduce stopping times 
	\begin{eqnarray*}
			T&=&\inf\{t\geq 0: X_t=Y_t\} \quad\text{and}\\
			T_m&=&\inf\{t\geq 0: \norm{X_t-Y_t}\leq 1/m \text{ or }\max\{\norm{X_t},\norm{Y_t}\}\geq m\}.
	\end{eqnarray*}
	Since the process $(X_t,Y_t)$ is non-explosive, we have $T_m\uparrow T$ for $m\uparrow\infty$. 
	We get
	\begin{eqnarray*}
	E\left[e^{c\,t}\;Q_t\right]&=&E\left[e^{c\,t}\;Q_t\,I_{t<T}\right]
	\ =\  \lim_{m\rightarrow\infty} E\left[e^{c\,{t\wedge T_m}}\;Q_{t\wedge T_m}\,I_{t<T_m}\right]
	\\&\leq& \liminf_{m\rightarrow\infty} E\left[e^{c\,{t\wedge T_m}}\;Q_{t\wedge T_m}\right].
	\end{eqnarray*}
	Fix $m\in\naturals$ and notice that the stopped process $(M^4_{t\wedge T_m})$ defined in Lemma
	\ref{thm3lem3}
	is a martingale. Using \eqref{thm2eqa2} and Lemma \ref{thm3lem3},
	we conclude
	\begin{eqnarray*}
	E\left[e^{c\,(t\wedge T_m)}Q_{t\wedge T_m}-Q_0\right]
	&\leq& E\left[\int_0^{t\wedge T_m} e^{c\, s} \,G(X_s,Y_s) \ ds\right]\ h(\delta)
	\\&\leq& E\left[\int_0^{t} e^{c\, s}\, G(X_s,Y_s) \ ds\right]\ h(\delta).
	\end{eqnarray*}
Assumption \ref{assLypDri} implies that 
	there is a constant $A\in(0,\infty)$ such that 
	\begin{eqnarray*}
	\sup_{t\in[0,\infty)} (\;E[V(X_t)]\ +\ E[V(Y_t)]\;)&<&A. 
	\end{eqnarray*}
\end{proof} 
 
 \begin{proof}[Proof of Corollary \ref{thm3cor1}]
	Let $x,y\in\hilbert$ with $\norm{x-y}_\alpha\leq \min\{1,R\}$.
By \eqref{normineq} and \eqref{fineq},
	\begin{eqnarray*}
		\norm{x-y}^p &\leq& \norm{x-y}_\alpha
\ \leq \ 2\,\phi^{-1}(\min\{1,R\})\;
	 f(\norm{x-y}_\alpha)\;(1+\epsilon \,V(x)+\epsilon\, V(y)).
	\end{eqnarray*}
	On the other hand, if $\norm{x-y}_\alpha>\min\{1,R\}$, then we get
	\begin{eqnarray*}
	\norm{x-y}^p \leq K (V(x)+V(y)) \leq \frac{K}{\epsilon f(\min\{1,R\})}
	f(\norm{x-y}_\alpha) (1+\epsilon \,V(x)+\epsilon\, V(y)).
	\end{eqnarray*}
	By \eqref{fineq}, $f(\min\{1,R\}) \  \geq\  \Phi(\min\{1,R\})/2 \ \geq \  \min\{1,R\}\; \phi(\min\{1,R\})/2$.
	Combining the bounds, we get for any $x,y\in\hilbert$,
	\begin{eqnarray}\label{b1}
		\norm{x-y}^p &\leq&	2\,\phi^{-1}(\min\{1,R\})\, \max\left\{1,\frac{K}{\epsilon\;\min\{1,R\}}\right\}\,
	 d_3(x,y). 
	\end{eqnarray}
	Using \eqref{b1} and Theorem \ref{thmThird}, we can conclude that 
	\begin{eqnarray}\label{aaa}
		\wDist^p(\mu p_t, \nu p_t)^p &\leq & 2\,\phi^{-1}(\min\{1,R\})\, \max\left\{1,\frac{K}{\epsilon\;\min\{1,R\}}\right\}\,
	 e^{-c\,t} \wDist_{d_3}(\mu , \nu )
	\end{eqnarray}
	for any $\mu,\nu\in\mathcal P_V$ and $t\geq 0$. Notice that Assumption
	\ref{assLypDri} implies that 
	\begin{eqnarray}
		\label{uniformbound}  \sup_{t\geq 0} \int V(y)\, (\delta_x
	p_t)(dy)<\infty \qquad \text{for any }
	x\in\hilbert.
	\end{eqnarray} 	
	In particular, \eqref{aaa} and \eqref{uniformbound} together imply that there
	is a constant $C\in(0,\infty)$ such that 
		\begin{eqnarray} \label{cauchy}
		\wDist^p(\delta_x p_m, \delta_x p_n)^p &=& \wDist^p(\delta_x p_{m-n} p_n, \delta_x p_n)^p
		\,\leq \, C\, e^{-c\,n}
	\end{eqnarray}
	for any integers $m>n>0$. We see that $(\delta_x p_n)_{n\in\naturals}$ is a
	Cauchy sequence w.r.t.\  $\wDist^p$. Moreover, the $L^p$ Wasserstein space is
	Polish and convergence w.r.t.\ $\wDist^p$ implies weak convergence. The
	Krylov-Bogolioubov criteria thus implies
	that there is a measure $\pi_0$ such that $\pi_0 p_1 = \pi_0$, cf. e.g.\
	\cite[Theorem 1.10]{hairerlecturenotes}. It is straightforward to check that
	$\pi:=\int_0^1 \pi_0 p_s\, ds$ is invariant w.r.t. $(p_t)$, cf.\ e.g.\
	\cite[Section 3]{komorowski2012central}.
	Moreover, Assumption \ref{assLypDri} implies that any invariant
	probability measure $\pi^\star$ satisfies $\pi^\star\in
	\mathcal{P}_V$, cf. e.g.\ \cite[Proposition 4.24]{hairer2006ergodic}, and thus
	\eqref{aaa} implies that $\pi$ is the only invariant measure.
\end{proof}   
\section{Applications}\label{secApplications}
We demonstrate the applicability of the results
from  Section \ref{secMain}. 
\subsection{Absolutely continuous measures w.r.t.\ a normal
distribution}\label{sec31}
\subsubsection{General setup} 
	Suppose that $\G$ is the covariance operator of a non-de\-gen\-er\-ate and
	centered normal distribution $\mathcal{N}(0,\G)$ on a separable Hilbert space
	$(\hilbert,\sProd{\cdot}{\cdot},\norm{\cdot})$, i.e. $\G$ is trace-class, symmetric and positive-definite. 
	 Define a probability measure $\pi$ by \eqref{pidef}, 
 	where $U:\hilbert\rightarrow \reals$ is a given potential which is bounded from below, Fréchet differentiable and for which $x\mapsto\nabla
 	U(x)$ is Lipschitz. We define $b$ in equation \eqref{eqMain} as 
 	$b(x)=-\G\nabla U(x)$.
 	The results  
 	from \cite{MR2358638} imply that $\pi$ is an invariant measure for $(p_t)$,
 	i.e.\ $\pi p_t= \pi$ for any $t\geq 0$. We now give sufficient conditions under which the results from Section
    \ref{secMain} are applicable in this context.

    \begin{remark}
    	The article  \cite{MR2358638} by Hairer, Stuart and Voss considers
    	two different SPDEs for which $\pi$ is a stationary distribution and which
    	 can both be used for sampling purposes. The first one is
    	 given by
    	\begin{eqnarray}\label{alternative}
    		d\tilde{X}_t &=& \Delta \tilde{X}_t\, dt \,-\, \nabla U(\tilde{X}_t)\, dt
    		\,+\, \sqrt{2} \, d\tilde{W}_t,
    	\end{eqnarray}
    	where $(\tilde{W}_t)$ is a cylindrical Wiener process over $\hilbert$.
    	The second one is given by \eqref{eqMain} and this is the one we study in
    	this article.
    	Formally, the latter equation is obtained from \eqref{alternative} by
    	``preconditioning''.
    	 The solutions for the equations behave quite differently: While
    	 \eqref{alternative} only admits mild solutions in general, strong
    	 solutions are possible for \eqref{eqMain}. Moreover,  as
    	 pointed out in \cite{MR2358638} under reasonable assumptions, the process
    	 $(\tilde{X}_t)$ is strong Feller 
    	  and it is 
    	 possible to apply classical Harris' theorems to study ergodic properties.
    	 In contrast to this, the process $(X_t)$ solving \eqref{eqMain} is not
    	 strong Feller in general and the study of ergodic properties is more involved.
    	  \end{remark} \par\medskip 
    	 
     Fix an orthonormal basis  
    $(\basis{k})_{k\in\mathbb{N}_+}$ of $\hilbert$ such that
   $\G \basis{k}=\lambda_k \basis{k}$ holds for a sequence
     $(\lambda_k)$ of positive reals
     satisfying $\sum_{k=1}^\infty \lambda_k <\infty$. For the sake of
     simplicity, we assume $\lambda_k\downarrow 0$ and $\lambda_1=1$.    
     Observe that Theorem \ref{thmSecond} and \ref{thmThird} still hold true, if we replace
     Assumptions \ref{assLipDrift} and \ref{assLocalNonContractive} by the
 	 slightly more general Assumptions \ref{assLipDrift2} and
     \ref{assLocalNonContractive2} respectively.
     \begin{assumption}\label{assLipDrift2} 
     There are constants $0\leq H_h<1$ and  $L_l,L_h, H_l\geq 0$ such that
\begin{eqnarray}\label{assLipDrift2eq1}
\sProd{\frac{x^h-y^h}{\norm{x^h-y^h}}}{b(x)-b(y)}  &\leq& H_l\; \norm{x^l-y^l}\ +\  H_h\; 
\norm{x^h-y^h}
\end{eqnarray}
for any $x,y\in\hilbert$ with $x^h\not=y^h$ and in the case $x^l\not=y^l$, we have
\begin{eqnarray}\label{assLipDrift2eq2}
\sProd{\frac{x^l-y^l}{\norm{x^l-y^l}}}{b(x)-b(y)}  &\leq& L_l\; \norm{x^l-y^l} \ +\ 
L_h\;
\norm{x^h-y^h}.
\end{eqnarray} 
\end{assumption}

\begin{assumption}\label{assLocalNonContractive2}
	There are $R\in(0,\infty)$ and $0\leq M<1$ such that
	\begin{align*}
		I_{x^l\not=y^l} \sProd{\frac{x^l-y^l}{\norm{x^l-y^l}}}{b(x)-b(y)}+ I_{x^h\not=y^h}
		\alpha \sProd{\frac{x^h-y^h}{\norm{x^h-y^h}}}{b(x)-b(y)} \leq M
		\norm{x-y}_\alpha
\end{align*}
 for any $x,y\in\hilbert$ with
	$\norm{x-y}_\alpha\geq R$.
\end{assumption}
In the following we focus on potentials  $U:\hilbert\rightarrow \reals$ of the form 
 \begin{eqnarray}
 	U(x) & = & \frac{a}{2}\norm{x}^2 \ +\  m(x),\label{potential}
 \end{eqnarray}
	where $a\geq 0$ and $m:\hilbert\rightarrow\reals$ satisfies:
	\begin{assumption} \label{assumptionM}
		The function $m$ is bounded from below and Fréchet differentiable.
		There is $L\geq 1$ such that
	\begin{eqnarray*}
		\norm{\nabla m(x)-\nabla m(y)} &\leq& L \; \norm{x-y} \qquad \text{holds true for any }
		x,y\in\hilbert. 
	\end{eqnarray*}
	\end{assumption}

\begin{lemma}\label{appLemma1}
	Let Assumption \ref{assumptionM} be true and define
	$n=\min\left\{k\in\naturals_{+}: \lambda_{k+1}<\frac{1}{2L}\right\}$.
	We consider the splitting 
	$\hilbert=\hilbert^l\oplus\hilbert^h$ with $\hilbert^l=\langle
	\basis{1},\ldots,\basis{n}\rangle$. In this setting, Assumption
	\ref{assLipDrift2} is satisfied with
	\begin{eqnarray*}
		H_l&=&H_h\ =\ {1}/{2}, \quad L_l\ =\ L_h\ =\ L, \quad \constA \ =\  2\;(1+ L) \quad \text{and} \quad \constB\ =\  2 \, L.
	\end{eqnarray*}	
\end{lemma}
\begin{proof} 
	Let  $x,y\in\hilbert$ with $x^h-y^h\not=0$. We have
		\begin{eqnarray*}
		-\sProd{x^h-y^h}{\G (\nabla U(x)-\nabla U(y))}
		&=& - a\sProd{x^h-y^h}{\G(x-y) } 
		\\&& -
		\sProd{x^h-y^h}{\G(\nabla m(x)-\nabla m(y))}.
		\end{eqnarray*}
		Observe that $- a\sProd{x^h-y^h}{\G(x-y) }\leq 0$. Using Cauchy-Schwarz, we get
		\begin{eqnarray*}
		\norm{\sProd{x^h-y^h}{\G(\nabla m(x)-\nabla m(y))}} &\leq&  \lambda_{n+1}\;L\; 
		\norm{x^h-y^h}\;\norm{x-y}\\ &\leq &
		{1}/{2}\; \norm{x^h-y^h}\; \norm{x-y}.
\end{eqnarray*}
This implies \eqref{assLipDrift2eq1} with $H_l=H_h={1}/{2}$. Inequality \eqref{assLipDrift2eq2} can be argued similarly. 
\end{proof}

\begin{lemma}\label{appLemma2}
	Assume that there is $\mathcal{R}>0$ such that $\nabla m(x)=0$ for any $\norm{x}\geq \mathcal{R}$. Then Assumption 
	\ref{assLocalNonContractive2} can be satisfied with $M=3/4$ and $R=8\,L\,\mathcal{R}$.
\end{lemma}

\begin{proof}
Let $x,y\in\hilbert$ with $\norm{x-y}_\alpha\geq R$. The statement is clear if $\min\{\norm{x},\norm{y}\}\geq \mathcal{R}$. Assume w.l.o.g.\
that $\norm{x}<\mathcal{R}$, $\norm{y}\geq \mathcal{R}$ and let $z\in\hilbert$ with $\norm{z}=\mathcal{R}$, then
\begin{eqnarray*}
&& I_{x^l\not=y^l} \sProd{\frac{x^l-y^l}{\norm{x^l-y^l}}}{b(x)-b(y)} \ +\  I_{x^h\not=y^h}
		\;\alpha\; \sProd{\frac{x^h-y^h}{\norm{x^h-y^h}}}{b(x)-b(y)} 
\\&\leq& \norm{ \left(\G(\nabla m(x)-\nabla m(z))\right)^l} \ +\ \alpha \;\norm{ \left(\G(\nabla m(x)-\nabla m(z))\right)^h} 
\\&\leq& L \, \norm{x-z} \ +\  \alpha\, \lambda_{n+1}\, L\, \norm{x-z}
\leq  2\, L \,  \mathcal{R} + 2\, (1+L) \,  \mathcal{R}\ \leq \ {3}/{4} \norm{x-y}_\alpha.
\end{eqnarray*}
\end{proof}

\begin{corollary}
	If the assumptions of Lemma \ref{appLemma1} and \ref{appLemma2} are satisfied, then Theorem \ref{thmSecond}
	holds with $4\,c\ \geq\  \exp\left(-32\,L^4\,\mathcal{R}^2\right)/{(1+L)}$.
\end{corollary}

We give a sufficient condition for the existence of a 
Lyapunov function.
\begin{lemma}\label{lemlypsamp} 
	Let Assumption \ref{assumptionM} be true and set $V(x) = 1+\norm{x}^2$.
	 If there are constants
	$b\in(0,\infty)$ and $0\leq c<1$ such that
	\begin{eqnarray}
		\norm{\nabla m(x)} &\leq& b \ +\  c\, \norm{x} \qquad \text{holds for any } x\in\hilbert,
	\end{eqnarray}
	then for any $0<\eta<1-c$ there is $C\in(0,\infty)$ such that 
	 Assumption \ref{assLypDri} is satisfied with $(V,C,\eta)$.
\end{lemma}

\begin{proof}
	 We have to find $C$ such that \eqref{eqLypGenerator} holds true for all $x\in\hilbert$. 
	 Observe that, 
	 \begin{eqnarray*}
	 	\mathcal{L}V(x) &=& 2\;\sProd{x}{-x-a\; \G\,x -\G \nabla m(x)}\ +\ \operatorname{trace}(\G)\\
	 	&\leq& 2\left(-\norm{x}^2 + b\, \norm{x} + c\, \norm{x}^2\right) \ +\  \operatorname{trace}(\G).
	 \end{eqnarray*}
	The claim follows since $c<1$.
\end{proof}
 
We see that Theorem \ref{thmThird} is applicable if the assumptions of Lemma \ref{lemlypsamp} are satisfied.
 
\subsubsection{Transition path sampling}
We present a concrete sampling context for which the results from the 
last subsection are applicable. We follow here \cite{MR2358638} and 
consider the $\reals^d$-valued SDE  
\begin{eqnarray}\label{sampleTargetSDE}  
	dX_t &=& -\nabla_{\reals^d}\, W(X_t)\,dt \ +\  dB_t, \qquad X_0=0, 
\end{eqnarray}
where $(B_t)$ is a $d$-dimensional Brownian motion. 
\begin{assumption}\label{assTPS} 
The potential $W:\reals^d\rightarrow \reals$  is  given by 
\begin{eqnarray*}
	W(x)&=& \frac{a}{2} \norm{x}^{2} \ +\  H(x),
\end{eqnarray*}
	with $a>0$ and $H:\reals^d\rightarrow\reals$ is a $C^4$ function for which 
	all $k$-fold partial derivatives, $k\in\{1,2,3,4\}$, satisfy
	\begin{align}\label{decay}
		\norm{\partial^k H(x)}\norm{x}^{k-2} \rightarrow 0 \text{ for }
		\norm{x}\rightarrow \infty.
	\end{align} 
\end{assumption}
Suppose that we are interested in the law $\pi$ of $(X_t)_{t\in[0,1]}$ conditioned on the event $X_1=0$.
We describe a typical setup for the above situation.
Set $\hilbert=L^2([0,1],\reals^d)$ and let $(\Delta_0, D_0)$ be the self-adjoint Laplacian with Dirichlet boundary condition, i.e.\  the
domain $D_0$ is given by all differentiable functions $f$,   
such that $f'$ is absolutely continuous with $f''\in L^2$ and such that $f(0)=f(1)=0$. 
Let  $\G=-\Delta_{0}^{-1}$. Observe that $\basis{k}=\sqrt{2}\, \sin(\pi k t)$, $k\in\naturals$,  is an orthonormal basis of $\hilbert$ satisfying 
$\G \basis{k} =\lambda_k \basis{k}$ with $\lambda_k=(\pi k)^{-2}$. In particular, the operator $\G$ is
trace-class, symmetric and positive definite on $\hilbert$.
It is well-known, that
the distribution of a standard Brownian Bridge on $\hilbert$ is a centered
normal distribution with covariance operator $\G$. Under Assumption \ref{assTPS} one can argue, using Girsanov's theorem and It{\^o}'s formula, that the
law $\pi$ of $(X_t)_{t\in[0,1]}$ conditioned on $X_1=0$ is given by \eqref{pidef},
with $U:\hilbert\rightarrow \reals$ defined by
\begin{eqnarray}
	U(x)&=&\frac{1}{2} \int_0^1 \Phi(x_s)\,ds \quad \text{and} \quad \Phi(x)\ =\ \norm{\nabla_{\reals^d} W(x)}^2+ 
\Delta_{\reals^d} W(x).
\end{eqnarray}
We refer the reader to 
\cite{MR2358638} for a detailed exposition. It is now a straightforward calculation to check
that Lemma \ref{lemlypsamp} is applicable, if Assumption \ref{assTPS} is satisfied.

\subsection{Finite-dimensional approximations} \label{secFinite}
In this work we focus on explicit contraction rates for the process \eqref{eqMain}. In the light of sampling applications one might ask, if it is possible to
 	make related statements about finite-dimensional approximations. We shortly argue, that this is indeed the case.
 	\par\smallskip
Suppose that Assumptions \ref{assLipDrift}, \ref{assNonDeg}, and \ref{assLocalNonContractive} are true. Let $\hilbert^l$ be of dimension $n\in\naturals_+$. 
Fix some $d>n$ and write $\hilbert^d=\langle
\basis{1},\ldots,\basis{d} \rangle$ for 
the subspace spanned by the first $d$ basis vectors. Given $x\in\hilbert$, we write $x^d$ for the orthogonal projection onto $\hilbert^d$.
Let $(X_t)$ be a solution of \eqref{eqMain} with $X_0=x_0$. 
A straightforward $d$-dimensional approximation $(\tilde{X}_t)$ is given by
the solution of the SDE
\begin{eqnarray} \label{eqMainApprox} 
	d\tilde{X}_t	&=& - \tilde{X}_t \ dt \ +\  b^d(\tilde{X}_t) \  dt \ + \  \sqrt{2}\  dW^d_t, \qquad \tilde{X}_0=x^d_0.
\end{eqnarray} 
A similar approximation is e.g.\ considered in \cite{arxiv1}.
Observe that the non-linearity $x\mapsto b^d(x)$ satisfies Assumptions \ref{assLipDrift} and \ref{assLocalNonContractive} on the space
$\hilbert^d$ with
the \emph{same constants} as $x\mapsto b(x)$ on $\hilbert$. In particular, we can apply Theorem \ref{thmSecond} to equation \eqref{eqMainApprox} and see
that the corresponding Markov kernels satisfy a Kantorovich contraction with a \emph{dimension-independent} and explicit contraction rate.
A related statement holds true for Theorem \ref{thmThird}. We remark that the unique invariant measure $\pi^d$ for
\eqref{eqMainApprox} does  in general not agree with the invariant measure $\pi$ of \eqref{eqMain}. A study of the 
approximation error can be found in \cite{arxiv1}. 
\par\smallskip
It might also be possible to use the presented results to make statements about the speed of convergence of
time-discrete approximations of \eqref{eqMainApprox}, 
 e.g.\ Euler approximations. There are at least two different approaches to this
 question: The first possibility is to implement a similar coupling strategy directly for Markov chains. We refer in this context 
 to the forthcoming work \cite{eberlePreparation}. The second possibility is to interpret the approximation
 as a \emph{perturbation} of the original equation, see 
 \cite{MR1756418,MR2012680,2014arXiv1405.0182P,2015arXiv150304123R,Dalalyan,2015arXiv150705021D} and the references therein.
 Nevertheless, the last question goes beyond the scope of this work.

\section*{Acknowledgement}
I want to thank my supervisor A. Eberle for suggesting the topic, for his support and helpful advice.
I am grateful for several fruitful discussions with A. Guillin on related topics. 
Moreover, I want to thank M. Hairer for his hospitality during a stay at the University of Warwick and 
for pointing out the connection of my work to the 2D Navier-Stokes equation.

\bibliographystyle{plain}
\bibliography{bib}

\end{document}